\documentclass[11pt]{article}
\usepackage[latin1]{inputenc}
\usepackage[T1]{fontenc}
\usepackage{lmodern}
\usepackage[a4paper]{geometry}
\usepackage{amsmath, amssymb}
\usepackage{amsthm}
\usepackage{dsfont}
\usepackage[english]{babel}
\usepackage{a4wide}
\usepackage{graphicx,color}
\usepackage{array}
\usepackage[nottoc, notlof, notlot]{tocbibind}
\usepackage{subfigure,enumitem}
\usepackage[colorlinks]{hyperref}
\usepackage{ulem}
\hypersetup{
linkcolor=blue,
citecolor=blue,
}

\DeclareMathOperator{\diam}{diam}

\title{Penalization of barycenters in the Wasserstein space}
\author{J\'er\'emie Bigot\footnote{J. Bigot is a member of Institut Universitaire de France.}, Elsa Cazelles \& Nicolas Papadakis \footnote{This work has been carried out with financial support from the French
State, managed by the French National Research Agency (ANR) in the frame
of the GOTMI project (ANR-16-CE33-0010-01). }  \\
\\  Institut de Math\'ematiques de Bordeaux et CNRS  (UMR 5251)   \\ Universit\'e de Bordeaux }

\begin{document}

\theoremstyle{plain}
\newtheorem{thm}{Theorem}[section]
\theoremstyle{plain}
\newtheorem{prop}{Properties}[section]
\theoremstyle{plain}
\newtheorem{hyp}{Assumption}[section]
\theoremstyle{plain}
\newtheorem{proposition}{Proposition}[section]
\theoremstyle{definition}
\newtheorem{defi}[thm]{Definition}
\theoremstyle{plain}
\newtheorem{lemma}[thm]{Lemma}
\newtheorem{cor}[thm]{Corollary}
\newtheorem{rmq}[thm]{Remark}
\theoremstyle{definition}
\newtheorem{ex}{Example}[section]
\newtheorem{appli}{Application}[section]

\newcommand{\E}{{\mathbb E}}
\newcommand{\R}{{\mathbb R}}
\newcommand{\N}{{\mathbb N}}
\newcommand{\Z}{{\mathbb Z}}
\renewcommand{\P}{{\mathbb P}}
\newcommand{\G}{{\mathbb G}}

\newcommand{\PP}{{\mathcal P}}
\newcommand{\HH}{{\mathcal H}}
\newcommand{\DD}{{\mathcal D}}
\newcommand{\BB}{{\mathcal B}}
\newcommand{\MM}{{\mathcal M}}

\newcommand{\bX}{\boldsymbol{X}}
\newcommand{\bY}{\boldsymbol{Y}}
\newcommand{\bnu}{\boldsymbol{\nu}}
\newcommand{\bmu}{\boldsymbol{\mu}}
\newcommand{\boldeta}{\boldsymbol{\eta}}
\newcommand{\bsigma}{\boldsymbol{\sigma}}
\newcommand{\bfun}{\boldsymbol{f}}

\numberwithin{equation}{section}

\maketitle

\thispagestyle{empty}

% REQUIRED
\begin{abstract}
In this paper, a regularization of  Wasserstein barycenters for random measures supported on $\R^{d}$ is introduced via convex penalization. The existence and uniqueness of such barycenters is first proved for a large class of penalization functions. The Bregman divergence associated to the penalization term is then considered to obtain a stability result on penalized barycenters.
This allows the comparison of data made of $n$ absolutely continuous probability measures, within the more realistic setting where one only  has  access to a dataset of random variables sampled from unknown distributions. The  convergence  of the penalized empirical barycenter of a set of $n$  iid random probability measures towards its population counterpart is finally analyzed.
This approach is  shown to be appropriate for the statistical analysis of  either discrete or absolutely continuous random measures. It also allows to construct, from a set of discrete measures, consistent estimators of population Wasserstein barycenters that are absolutely continuous.
\end{abstract}

%% REQUIRED
%\begin{keywords}
%  Wasserstein space, Fr\'echet mean, Barycenter of probability measures, Convex penalization, Regularization, Bregman divergence
%\end{keywords}
%
%% REQUIRED
%\begin{AMS}
% 62G07, 62G20
%\end{AMS}

\section{Introduction}

In this work, we consider the Wasserstein distance $W_2$ associated to the quadratic cost for the comparison of probability measures (see e.g.\  \cite{villani2003topics} for a thorough introduction on the topic of Wasserstein spaces and optimal transport). Let $\Omega$ be a convex subset of $\R^d$ and  $\PP_2(\Omega)$ be the set of probability measures supported on $\Omega$ with finite second order moment. As introduced in \cite{agueh2011barycenters}, an empirical Wasserstein barycenter  of set of $n$ probability measures $\nu_{1},\ldots,\nu_{n}$ (not necessarily random) in $\PP_2(\Omega)$ is defined as a minimizer of
\begin{equation}
\mu \longmapsto  \frac{1}{n} \sum_{i=1}^{n} W_2^2(\mu,\nu_{i}), \mbox{ over }  \mu \in \PP_2(\Omega). \label{eq:pbempBar}
\end{equation}
The Wasserstein barycenter  corresponds to the notion of empirical Fr\'echet mean \cite{fre} that is an extension of the usual Euclidean barycenter to nonlinear metric spaces.

More generally, by introducing a probability distribution $\P$ on the space of probability measures $\PP_2(\Omega)$, we can define  a Wasserstein barycenter $\mu_{\P}$ of the distribution $\P$ as
\begin{equation}
\mu_{\P} \in \underset{\mu\in\PP_2(\Omega)}{\text{argmin}}\ \int_{\PP_2(\Omega)}W_2^2(\mu,\nu)d\P(\nu).  \label{eq:pbpopBar}
\end{equation}
%Accordingly, $\mu_{\P}$ can be referred to as the structural mean.
For a discrete distribution  $\P_{n} =\frac{1}{n}\sum\delta_{\nu_i}$ on $\PP_2(\Omega)$, one has that $\mu_{\P_{n}}$ corresponds to the empirical Wasserstein barycenter defined in \eqref{eq:pbempBar}. In the setting where $\bnu_{1},\ldots,\bnu_{n}$ are independent and identically distributed (iid) random probability measures sampled from a distribution $\P$, the barycenter $\mu_{\P}$ is referred to as the population counterpart of  $\mu_{\P_{n}}$.

In this work, for the purpose of obtaining a regularized Wasserstein barycenter, we introduce a  convex penalty function $E$ in the optimization problem \eqref{eq:pbpopBar}   by considering the convex minimization problem
\begin{equation}\label{penalized_problem}
\min_{\mu\in \PP_2(\Omega)} \ \int W_2^2(\mu,\nu)d\P(\nu)+\gamma E(\mu)
\end{equation}
where $\gamma>0$ is a penalization parameter and $\P$ is any distribution on $\PP_2(\Omega)$ (possibly discrete or not). A first contribution of this paper is then to prove the existence and uniqueness of the minimizers of \eqref{penalized_problem}, called penalized (or regularized) Wasserstein barycenters, for a large class of penalization functions, and to study their stability with respect to the distribution $\P$.

 Introducing a penalization term in the definition \eqref{eq:pbempBar} of an empirical Wasserstein barycenter   is a way to incorporate some priori knowledge on the behavior of its population counterpart. In particular, we are interested in the case where   the dataset at hand is composed of $n$   discrete measures $\bnu_{p_{1}},\ldots,\bnu_{p_{n}}$ obtained from random observations
\begin{equation}
\bX = (\bX_{i,j})_{1 \leq i \leq n; \; 1 \leq j \leq p_{i} }, \label{eq:bX}
\end{equation}
organized in the form of $n$ experimental units, such that $\bnu_{p_{i}}$ is defined  by 
\begin{equation}
\bnu_{p_{i}} = \frac{1}{p_{i}} \sum_{j=1}^{p_{i}} \delta_{\bX_{i,j}}. \label{eq:discrete}
\end{equation}
 Typically,  for each $i=1,\ldots,n$, the random variables $\bX_{i,1},\ldots,\bX_{i,p_{i}}$ are iid observations in $\R^{d}$ generated from an absolutely continuous (a.c.) measure $\bnu_{i}$ (that is also random). In the paper, absolute continuity is always understood with respect to the Lebesgue measure $dx$ on $\R^{d}$.

 In this setting, an empirical Wasserstein barycenter of the discrete measures $\bnu_{p_{1}},\ldots,\bnu_{p_{n}}$ is generally irregular (and even not unique). Moreover, it poorly represents the Wasserstein barycenter of the  a.c.\ measures  $(\bnu_i)_{i=1,\ldots,n}$ (which is unique and smooth as proved in \cite{agueh2011barycenters}). A second contribution of this paper is then to show that introducing a penalization term in the computation of Wasserstein barycenters of  discrete measures allows to construct a consistent estimator of an a.c.\ population barycenter in the asymptotic setting where both $n$ and $\min_{1 \leq i \leq n} p_{i}$ tend to infinity.

Let us underline that we mainly focus on penalization functions $E$ that enforce  the minimizer of \eqref{eq:pbpopBar} to be an a.c.\ measure with a smooth probability density function (pdf). % A first impulse is to recover the barycenter of an underlying distribution of a sample of observations generated from an a.c.\ measure. Another motivation is the need to handle possible outliers in the dataset.
In this case, we propose to control the penalized Wasserstein barycenter %{\color{blue} (e.g.\ with respect to its population counterpart) }
in term of the  Bregman divergence associated to the penalty function $E$. % when both $\min_{i} p_i$ the number of observations by measures and $n$ the number of measures diverge.
Bregman divergences  have been proved to be relevant  measures of discrepancy between a.c.\ probability measures (e.g.\  in information geometry \cite{Amari2000}). Their constructions also cover a large range of known divergences for different penalty functions. It is therefore natural to use Bregman divergences to  compare penalized empirical Wasserstein barycenters of discrete measures with an a.c.\ population barycenter. To the best of our knowledge, this has not been considered so far.% Besides, these allow us to control the error between the population barycenter and its penalized empirical version with respect to the number of samples and the penalization parameter $\gamma$. As the Wasserstein metric is difficult to bound from below, rates of convergence of random measures towards their population counterpart are delicated to obtain (see e.g. \cite{fournier:hal-00915365}, \cite{del2017central}), which can be harder for combined measures like barycenters.

% to compare regular Wasserstein barycenters} as they 

%We focus on first-order statistics methods for the purpose of estimating a population mean (or barycenter) measure or density function from a dataset of random variables  $\bX = (\bX_{i,j})_{1 \leq i \leq n; \; 1 \leq j \leq p_{i} }$  organized in the form of $n$ experimental units, such that $\bX_{i,1},\ldots,\bX_{i,p_{i}}$ are iid observations in $\R^{d}$ sampled from a measure (discrete of absolutely continuous) $\nu_{i}, 1\leq i\leq n$. 

%For probability measures supported on the real line, computing a Wasserstein barycenter simply amounts to averaging the quantile functions of the $\nu_{i}$'s (see e.g.\ Section 6.1 in \cite{agueh2011barycenters}).  
%The Wasserstein distance belongs to a larger class of divergences named optimal transportation costs. They aim at measuring the amount of mass transferred from a distribution $\mu$ onto another one $\nu$ through an optimal transport plan which gives the behavior of the masses' transferences. In particular these divergences takes into account the support of the distributions when a $\mathbb{L}_2$ distance can not. The Wasserstein distance for the $2$-euclidean norm as a cost is particularly interesting as it metrizes the weak convergence of measures.
\newpage
\subsection{Related work in the literature}

\paragraph*{Statistical inference using optimal transport.} 
The penalized barycenter problem is motivated by the nonparametric method introduced in \cite{burger2012regularized} for the classical density estimation problem from discrete samples.  It is based on a variational regularization approach involving the Wasserstein distance as a data fidelity term. However, the adaptation of this work for the penalization of Wasserstein barycenter has not been considered so far. 

%For barycenter in the case $d=1$, 
\paragraph*{Consistent estimators of population Wasserstein barycenters.}
Tools from optimal transport are used in \cite{Pana15} for the registration of multiple point processes which represent repeated observations organized in samples from independent subjects or experimental units. The authors  in \cite{Pana15} proposed a consistent estimator of the %{\bf \color{blue}population \eqref{def_J} pas encore d\'efini} 
population Wasserstein barycenter of multiple point processes in the case $d=1$, and an extension of their methodology for $d \geq 2$ is considered in \cite{Pana17}. Their method contains two steps. A kernel smoothing is first applied to the data  which leads to a set of  a.c.\ measures from which an empirical Wasserstein barycenter is computed in a second step. Our approach thus differs from  \cite{Pana15,Pana17}  since  we directly include the  smoothing step  in the computation of a Wasserstein barycenter via  the penalty function $E$ in \eqref{penalized_problem}.   Also notice that estimators of population Wasserstein barycenter are shown to be consistent  for the Wasserstein metric $W_2$ in   \cite{Pana15,Pana17}, whereas we  prove the consistency of our approach for metrics in the space of pdf supported on $\R^{d}$. Finally,  rates of convergence (for the Wasserstein metric $W_2$) of empirical Wasserstein barycenters computed from discrete measures, supported on the real line only, are discussed in \cite{Pana15,BGKL18}.

%{\bf\color{blue} Pourquoi la  phrase suivante a-t-elle \'et\'e comment\'ee? on ne peut pas au moins reformuler en disant que l'a g\'en\'eralisation a d>1 est loin d'etre directe?} 
% Moreover, the quantile based methodology in \cite{Pana15} cannot be extended beyond the one-dimensional case as the Wasserstein distance then boils down to the $\mathbb{L}_2$ metric of the quantile functions of the measures.

\paragraph*{Generalized notions of Wasserstein barycenters.}

%As explained above, the notion of barycenter in the Wasserstein space for a finite set of probability measures has been first introduced in \cite{agueh2011barycenters}.
A detailed characterization of empirical Wasserstein barycenters  in terms of existence, uniqueness and regularity for probability measures with support in $\R^d$ is given in  the seminal paper \cite{agueh2011barycenters}. The relation of such barycenters with the solution of the multi-marginal problem is also studied in both \cite{agueh2011barycenters} and \cite{pass2013optimal}.
  %using a dual formulation of problem \eqref{eq:pbempBar} and arguments from convex analysis.
%There exists also a link between Wasserstein barycenters and the multi-marginal problem in optimal transport (with possibly infinitely many marginals)as studied in \cite{agueh2011barycenters} and \cite{pass2013optimal}.
The notion of Wasserstein barycenter has  been first generalized  in \cite{gouic2015existence}, by establishing  existence, uniqueness and consistency for random probability measures  supported on a locally compact geodesic space.
% $(G,d)$ by considering the minimizers of
%$
%\mu \mapsto \int W_2^2(\mu,\nu)d\P(\nu)  
%$
%where $\mu$ is a probability measure on $G$ with finite second moment, $W_2$ is the $2$-Wasserstein distance associated to the metric $d$, and $\P$ is a distribution on probability measures. When $\P = \P_{n}^\nu = \frac{1}{n}\sum_{i=1}^n\delta_{\nu_i}$, this approach allows to define a generalized notion of empirical Wasserstein barycenter. 
The more general case of probability measures supported on  Riemannian manifolds  has then been  studied in \cite{KimPass}. Subsequently, trimmed barycenters in the Wasserstein space  $\PP_2(\R^{d})$ have been introduced in  \cite{alvarez2015wide} for the purpose of  combining informations from different experimental units in a parallelized or distributed estimation setting. The framework of optimal transport has been recently  adapted for nonnegative measures supported on a compact subset in $\R^d$ with different masses, independently by \cite{chizat2016interpolating} and \cite{figalli2010new}.
However, in all these papers, incorporating regularization into the computation of Wasserstein barycenters has not been considered, which is of interest when the data are irregular probability measures.  

\paragraph*{Penalization of the transport map.} Alternatively, regularized barycenters may be obtained  by  adding a convex regularization on optimal  transport plans (that is on $\pi$ in \eqref{W_2}) when computing the Wasserstein distance between probability measures. This approach leads to the notion of regularized transportation problems and Wasserstein costs. It has  recently gained popularity in the literature of image processing and machine learning and has been considered to compute smoothed Wasserstein barycenters of discrete measures  \cite{ferradans2014regularized,cuturi2013fast}. 
% j'ai vire les autres papiers qui ne sont que des reecritures algorithmiques du papier original de Cuturi et Doucet
%{\color{red} que vient faire le papier qui suit dans la partie penalisation de transport map?:Contributions also include the design of algorithms to solve multimarginal optimal transport problems \cite{carlier2014numerical}.  } 
%Ils le motivent pas comme ca: In  these works, regularization of the transport plan is motivated by the  need  to accelerate the  computation of the Wasserstein distance  between probability measures supported on $\R^{d}$ with $d \geq 2$.
Such penalizations acting on transport plans nevertheless have an indirect influence on the regularity of the Wasserstein barycenter. % is nevertheless not explicitely monitored up to now.
 
\subsection{Contributions and structure of the paper} \label{sec:main}

The results of the paper are organized as follows.

\begin{enumerate}%[leftmargin=*]%[wide, labelwidth=!, labelindent=0pt]
\item[-] In Section \ref{sec:def}, we introduce the main notations and definitions, and we present a key result called subgradient's inequality on which a large part of the developments in the paper lean.
\item[-] In Section \ref{sec:exist}, we analyze the existence, uniqueness and stability of penalized Wasserstein barycenters \eqref{penalized_problem} for various  penalty functions $E$, any  parameter $\gamma>0$ and for either a discrete distribution $\P_n$ supported on $\PP_2(\Omega)$ or its population counterpart $\P$. We also prove a stability result allowing to compare the case of data made of $n$ a.c.\ probability measures $\nu_{1},\ldots,\nu_{n}$, with the more realistic setting where we have only access to a dataset of random variables as in \eqref{eq:bX}.
\item[-] In Section \ref{sec:conv},  we derive 
convergence properties of empirical penalized barycenters toward their population counterpart in the asymptotic setting where the number of measures $n$ tends to infinity. 
These convergence results are obtained with respect to the Bregman divergence associated to the penalization function of the minimization problem \eqref{penalized_problem}. 
%for the Bregman divergence associated to the penalization term,
%and $\gamma = \gamma_n$ is let going to zero.  
In this context, we demonstrate  that the bias term (as classically referred to in nonparametric statistics) converges to zero when $\gamma \to 0$. We also show (for $d=1$ and with additional regularity assumptions for $d\geq 2$) that  the variance term converges to $0$ when $\gamma = \gamma_n$ is let going to zero and $\lim_{n\to\infty} \gamma_n^2n=+\infty$. 
%\item[-] To illustrate the benefits of penalized barycenters for data analysis,  we propose to use in Section \ref{sec:num}  efficient minimization algorithms  for the computation of penalized  barycenters as well as a selection strategy for the parameter $\gamma$. This approach is finally applied to the statistical analysis of simulated and real data sets in $\PP_2(\R)$ and $\PP_2(\R^2)$.
%Efficient algorithms  for the computation of penalized empirical barycenters as well as a parameter selection strategy are proposed in together with  applications on simulated and real data sets. 

\item[-] In Section \ref{sec:discuss}, we conclude the paper by a discussion on the consistency of our approach when the data at hand are discrete measures as in \eqref{eq:discrete}.

\item[-] The proofs of the main results are gathered in Appendix \ref{App_proof}. % {\color{blue} Additional definitions are given in Appendix \ref{App_def}, while  the proofs of the main results are} gathered in Appendix \ref{App_proof}.
\end{enumerate}

Finally, it should be mentioned that the computational aspects on the numerical approximation of penalized Wasserstein barycenter (as introduced in this work)  are the subject of the companion paper \cite{BCP18}, where we also tackle the issue of choosing the regularization parameter $\gamma$ in \eqref{penalized_problem}. In \cite{BCP18}, we also discuss the consistency of smooth Wasserstein barycenters as proposed in \cite{cuturi2013fast} using entropically regularized optimal transport. 

\section{Definitions, notation and first results} \label{sec:def}

\subsection{Wasserstein distance, Kantorovich's duality and Wasserstein bary\-centers}

For $\Omega$ a convex subset of $\R^d$, we denote by $\MM(\Omega)$ the space of bounded Radon measures on $\Omega$ and by $\PP_2(\Omega)$ the set of probability measures over $(\Omega,\mathcal{B}(\Omega))$ with finite second order moment, where $\mathcal{B}(\Omega)$ is the $\sigma$-algebra of Borel subsets of $\Omega$. In particular, $\PP_2(\Omega)\subset\MM(\Omega)$.
\begin{defi}
The \textit{Wasserstein distance} $W_2(\mu,\nu)$ is defined for $\mu,\nu\in\PP_2(\Omega)$ by
\begin{equation}
\label{W_2}
W_2^2(\mu,\nu)=\underset{\pi}{\inf}\iint_{\Omega\times\Omega}\vert x-y\vert^2d\pi(x,y),
\end{equation}
where the infimum is taken over all probability measures $\pi$ on the product space $\Omega\times\Omega$ with respective marginals $\mu$ and $\nu$. %Note that $\mu\mapsto W_2^2(\mu,\nu)$ is a convex function.
\end{defi}
The well known Kantorovich's duality theorem leads to another formulation of the Wasserstein distance.
\begin{thm}[Kantorovich's duality]
\label{Kantorovich}
For any $\mu,\nu\in\PP_2(\Omega)$, one has that
\begin{equation}
\label{W_2_kanto}
W_2^2(\mu,\nu)=\underset{(\phi,\psi)\in C_W}{\sup}\int_{\Omega}\phi(x)d\mu(x)+\int_{\Omega}\psi(y)d\nu(y),
\end{equation}
where $C_W$ is the set of all measurable functions $(\phi,\psi)\in\mathbb{L}_1(\mu)\times\mathbb{L}_1(\nu)$ satisfying \begin{equation}
\phi(x)+\psi(y)\le\vert x-y\vert^2, \label{eq:condW}
\end{equation}
for $\mu$-almost every $x\in\Omega$ and $\nu$-almost every $y\in\Omega$. %A couple $(\phi,\psi)\in C_W$ that attains the supremum is called an optimal couple for $(\mu,\nu)$.
\end{thm}
For a detailed presentation of the Wasserstein distance and Kantorovich's duality, we refer to  \cite{villani2003topics,villani2008optimal}. 
For  $\mu, \nu \in \PP_2(\Omega)$, we denote by $\pi^{\mu,\nu}$ an \textit{optimal transport plan}, that is a solution of \eqref{W_2} satisfying $W_2^2(\mu,\nu)=\iint\vert x-y\vert^2d\pi^{\mu,\nu}(x,y)$. Likewise a pair $(\phi^{\mu,\nu},\psi^{\mu,\nu})\in\mathbb{L}_1(d\mu)\times\mathbb{L}_1(d\nu)$ achieving the supremum in \eqref{W_2_kanto} (under the constraint  $\phi^{\mu,\nu}(x)+\psi^{\mu,\nu}(y)\le\vert x-y\vert^2$)
stands for the \textit{optimal couple} in the Kantorovich duality formulation of the Wasserstein distance between $\mu$ and $\nu$.

%We can then define the Wasserstein barycenter, which corresponds to the notion of empirical Fr\'echet mean \cite{fre} that is an extension of the usual Euclidean barycenter to nonlinear metric spaces. As introduced in \cite{agueh2011barycenters}, an empirical Wasserstein barycenter $\bar{\nu}_{n}$ of set of $n$ probability measures $\nu_{1},\ldots,\nu_{n}$ (not necessarily random) in $\PP_2(\Omega)$ is defined as a minimizer of
%\begin{equation}
%\mu \mapsto  \frac{1}{n} \sum_{i=1}^{n} W_2^2(\mu,\nu_{i}), \mbox{ over }  \mu \in \PP_2(\Omega). \label{eq:pbempBar}
%\end{equation}

\subsection{Penalized barycenters of a random measure}
Throughout the paper, we use  bold symbols $\bnu, \bX, \bfun, \ldots$ to denote random objects. 
A probability measure $\bnu$ in $\PP_2(\Omega)$ is said to be random if it is sampled from a distribution $\P$ on $(\PP_2(\Omega), \BB \left( \PP_2(\Omega)\right)$, where $\BB \left( \PP_2(\Omega) \right)$ is the  Borel $\sigma$-algebra  generated by the topology induced by the distance $W_{2}$. Then, we introduce a Wasserstein distance between distributions of random measures (see \cite{gouic2015existence} and \cite{alvarez2015wide} for similar concepts), and the notion of Wasserstein barycenter of a random probability measure $\bnu$. 
\begin{defi}
Let $W_2(\PP_2(\Omega))$ be the space of distributions $\P$ on $\PP_2(\Omega)$ (endowed with the Wasserstein distance $W_{2}$) such that for some (thus for every) $\mu\in\PP_2(\Omega)$
\[\mathcal{W}_2^2(\delta_{\mu},\P):=\E_{\P}(W_2^2(\mu,\bnu))=\int_{\PP_2(\Omega)}W_2^2(\mu,\nu)d\P(\nu)< + \infty,\]
where $\bnu\in\PP_2(\Omega)$ is a random measure with distribution $\P$ and $\delta_{\mu}$ denotes the Dirac measure at point $\mu$. The Wasserstein barycenter of a random probability measure with distribution  $\P\in W_2(\PP_2(\Omega))$ is defined as a minimizer of
\begin{equation}
\mu\in\PP_2(\Omega)\mapsto\mathcal{W}_2^2(\delta_{\mu},\P)=\int_{\PP_2(\Omega)}W_2^2(\mu,\nu)d\P(\nu)  \mbox{ over } \mu \in  \PP_2(\Omega). \label{eq:defbarP} %=\int_{\PP_2(\Omega)}W_2^2(\mu,\nu_{\omega})d\P(\omega)
\end{equation}
\end{defi}
Thanks to the results in \cite{gouic2015existence}, there exists a minimizer of \eqref{eq:defbarP}, and thus the notion of Wasserstein barycenter of a random probability measure is well defined. 
%\NP{il manque une transition ici:}
As explained before, our goal is to study a penalized version of this barycenter. Hence, 
throughout the paper the following assumptions are made on the penalizing function $E$.
\begin{hyp} \label{hyp_E}
A penalizing function  $E: \PP_2(\Omega) \to \R_{+}$ is a proper and lower semicontinuous function (for the Wasserstein distance $W_{2}$)  that is strictly convex on its domain
\begin{equation}
\DD(E) = \left\{ \mu \in \PP_2(\Omega) \mbox{ such that } E(\mu) < +\infty  \right\}. \label{eq:domE}
\end{equation}
\end{hyp}
In this paper, we will often rely on the class of relative $G$-functionals (see Chapter 9, Section 9.4 of \cite{ambrosio2008gradient}) defined below.
\begin{defi}\label{defini_Gfun}
The relative $G$-functional with respect to (w.r.t) a given positive measure $\lambda \in \MM(\Omega)$ is the function $E: \PP_2(\Omega) \to \R_{+}$ defined by
\begin{equation} \label{eq:Gfun}
E(\mu) = \left\{\begin{array}{ll}
\displaystyle\int_{\Omega} G\left(\frac{d\mu}{d\lambda}(x)\right) d\lambda(x), & \mbox{if}\  \mu\ll \lambda \\
+\infty & \mbox{otherwise,}
\end{array}\right.
\end{equation}
where $G:[0,+\infty)\to [0,+\infty]$ is a proper, lower semicontinuous and strictly convex function with superlinear growth.

Thanks to Lemma 9.4.3 in \cite{ambrosio2008gradient},  a relative $G$-functional is a lower semicontinuous function for the Wasserstein distance $W_{2}$, so that it satisfies Assumption \ref{hyp_E}.  
\end{defi}

When $\lambda$ is the Lebesgue measure on $\Omega\in\R^d$, choosing such a penalizing function  enforces the Wasserstein barycenter to be a.c.  Hence, 
a typical example of  penalizing function satisfying Assumption \ref{hyp_E} is the negative entropy \cite{burger2012regularized} (see e.g.\ Lemma 1.4.3 in \cite{MR1431744}) defined as (assuming e.g. that $\Omega$ is compact)
\begin{equation}\label{def:negative_entropy}
E_e(\mu) = \left\{\begin{array}{ll}
\displaystyle\int_{\Omega} (f(x) (\log(f(x))-1)+1) dx, & \mbox{if $\mu$ admits a density $f$ with respect}\vspace{-0.2cm}\\
& \mbox{to the Lebesgue measure $dx$ on $\Omega$,}\vspace{0.2cm} \\
+\infty & \mbox{otherwise.}
\end{array}\right.
\end{equation}
%which enforces the barycenter to be a.c.\ with respect to the Lebesgue measure on $\R^{d}$. % It is particularly interesting since if at least one measure $\nu_i$ is a.c., the Wasserstein barycenter is itself a.c.\ (see \cite{agueh2011barycenters}). %This is of interest especially since the theory tell us that if at least one measure $\nu_i$ is a.c., then its Wasserstein barycenter will be a.c.
It is of interest to use the negative entropy as a penalizing function when one has only  access to discrete observations, that is in the setting where each $\bnu_i$ is a discrete measure of the form \eqref{eq:discrete}. Indeed in this case, the resulting Wasserstein barycenter minimizing \eqref{eq:pbempBar} will not necessary be a.c.\ (unless it it penalized) whereas we are interested in recovering a density from discrete measures. \\ %In this case, a discrete barycenter will not represent in a satisfying way the underlying measures $\nu_i$.

Penalized Wasserstein barycenters of a random measure $\bnu \in \PP_2(\Omega)$ are then defined as follows.
\begin{defi} \label{defini_J}
Let $E$ be a penalizing function satisfying Assumption  \ref{hyp_E}. For a distribution  $\P\in W_2(\PP_2(\Omega))$ and a penalization parameter $\gamma \geq 0$,  the functional $J_{\P}^{\gamma} : \PP_2(\Omega) \to \R_{+}$ is defined as
\begin{equation}
\label{def_J}
J_{\P}^{\gamma}(\mu) =  \int_{\PP_2(\Omega)} W_2^2(\mu,\nu)d\P(\nu)+\gamma E(\mu), \; \mu \in \PP_{2}(\Omega).
\end{equation}
If it exists, a minimizer $\mu_{\P}^{\gamma}$ of $J_{\P}^{\gamma}$ is called a penalized Wasserstein barycenter of the random measure $\bnu$ with distribution  $\P$. % (or a penalized population Wasserstein barycenter). 

In particular, if $\P$ is the discrete (resp. empirical) measure defined by
$\P=\P_n=\frac{1}{n}\sum_{i=1}^n\delta_{\nu_i} $ (resp. $\P_n=\frac{1}{n}\sum_{i=1}^n\delta_{\bnu_i}$) where each $\nu_i \in \mathcal{P}_2(\Omega)$ (resp. $\bnu_i \in \mathcal{P}_2(\Omega)$ random), then $J_{\P}^{\gamma}$ becomes
\begin{equation}
\label{def_Jn}
J_{\P_n}^{\gamma}(\mu) =  \frac{1}{n}\sum_{i=1}^n W_2^2(\mu,\nu_i)+\gamma E(\mu).
\end{equation}
Note that $J_{\P}^{\gamma}$ is strictly convex on $\DD(E)$ by Assumption \ref{hyp_E}.
\end{defi} 

\subsection{Subgradient's inequality}
In order to analyze the stability of the minimizers of $J_{\P}^{\gamma}$ with respect to the distribution $\P$, the notion of Bregman divergence  related to a  sufficiently smooth penalizing function $E$  will be needed. To simplify the presentation, we shall now restrict our analysis to relative $G$-functionals \eqref{eq:Gfun}.

\begin{defi}[Subgradient]
Let  $J: \PP_2(\Omega)  \to \R$ be a convex, proper and  lower semicontinuous function. 
Any subgradient $\xi\in \partial J(\mu)$ of $J$ at $\mu \in \DD(J)$  satisfies the inequality
\begin{equation}
\label{def_subgradient}
J(\nu) \ge J(\mu)+ \langle  \xi , \nu-\mu \rangle \ \text{for every} \ \nu \in \PP_2(\Omega),
\end{equation}
and the linear form in the right-hand side of \eqref{def_subgradient}  is understood as
$$
 \langle \xi, \nu-\mu \rangle = \int_{\Omega}  \xi(x)(d\nu(x) - d\mu(x)).
$$
If the function is differentiable, then the subdifferential $\partial J(\mu)$ is a singleton, and thus we have $\partial J(\mu)=\{\nabla J(\mu)\}$, the gradient of $J$ at point $\mu$.
\end{defi}

In what follows, we will consider subgradients for two different purposes: (i) to define a Bregman divergence with respect to $E$ and (ii) to obtain the main result of this section that involves subgradient of the Wasserstein distance.

\begin{defi}
A penalizing function $E$ is said to be a smooth relative $G$-func\- tional if the function $G$ is differentiable on $[0,+\infty)$. 
\end{defi}
From Definition \ref{defini_Gfun}, we directly have the following proposition.
\begin{proposition}
Let  $E$ be a smooth relative $G$-functional. 
We denote by $\nabla E(\mu)$ the subgradient of $E$ at $\mu \in \DD(E)$ taken as
$$
\nabla E(\mu)(x) = \nabla G\left(\frac{d\mu}{d\lambda}(x)\right), \; x \in \Omega.
$$
\end{proposition}

\begin{defi}[Bregman divergence]
Let  $E$ be a smooth relative $G$-functional. For $\mu,\nu\in \DD(E) \subset \PP_2(\Omega)$ the (symmetric) Bregman divergence related to $E$ is defined by
\begin{equation}
d_E(\mu,\nu)=\langle \nabla E(\mu) - \nabla E(\nu), \mu-\nu \rangle. \label{eq:dE}
\end{equation}
\end{defi}

\begin{rmq}
More generally, the Bregman divergence between $\mu$ and $\nu$ related to a convex functional $J: \PP_2(\Omega)  \to \R$ is defined for two particular subgradients $\xi\in\partial J(\mu)$ and $\kappa\in\partial J(\nu)$ by
$$d_J^{\xi,\kappa}(\mu,\nu)=\langle \xi-\kappa,\mu-\nu\rangle.$$
\end{rmq}

To illustrate the above definitions, let us assume that $\lambda$ is the Lebesgue measure $dx$, and consider two a.c.\  measures $\mu = \mu_{f}$ and $\nu = \nu_{g}$ with density $f$ and $g$. An example of a smooth relative $G$-functional is the case where $G(u) = u^2 / 2$ for which $E(\mu_{f}) = \frac{1}{2} \| f \|^2_{\mathbb{L}^2(\Omega)}= \frac{1}{2} \int_{\Omega} |f(x)|^2 dx $,
$$
\nabla E(\mu_{f})(x) =  f(x), \; \nabla E(\nu_{g})(x) = g(x)  \quad \mbox{ and } \quad  d_E(\mu_{f},\nu_{g})=  \int_{\Omega} (f(x) - g(x))^2 dx.
$$

\begin{rmq} \label{rmq:entropy}
It should be noted that the case where $E$ is the negative entropy $E_e$ defined in \eqref{def:negative_entropy} is critical. Indeed, the negative entropy is obviously a relative $G$-functional with $G(u) = u (\log(u)-1) +1$ and $\lambda = dx$. However, as  this function is not differentiable at $u = 0$, it does not lead to  a smooth relative $G$-functional. To use such a penalizing function, it is necessary to restrict the analysis of penalized Wasserstein barycenters to the set of a.c.\ measures in $\PP_2(\Omega)$ with densities that are uniformly bounded from below by a positive constant  $\alpha$ on the set $\Omega$. In this setting, we have that
$$
\nabla E_e(\mu_{f})=\log(f(x)) \quad \mbox{and} \quad \nabla E_e(\nu_{g})=\log(g(x)), \; x \in \Omega,
$$
and the  Bregman divergence is the symmetrized Kullback-Leibler divergence
$$
d_{E_e}(\mu_{f},\nu_{g}) = \int_{\Omega}  (f(x) - g(x)) \log\left(\frac{f(x)}{g(x)}\right) dx,
$$
where $f(x) \geq \alpha$ and $g(x) \geq \alpha$ for all $x \in \Omega$. \\
\end{rmq}
Then, a key result to study the stability of penalized Wasserstein barycenters with respect to the distribution $\P$ is stated below. It involves a subgradient $\phi$  of the Wasserstein distance. As detailed in the proof given in the Appendix \ref{App_subgrad}, this subgradient   corresponds to the Kantorovich potential introduced in Theorem \ref{Kantorovich}.
\begin{thm}[Subgradient's inequality]
\label{thm_subgradient}
Let  $E$ be a smooth relative $G$-func\-tional and thus satisfying Assumption  \ref{hyp_E}. Let $\nu$ be a probability measure in $\PP_2(\Omega)$, and define the functional  
\[J: \mu\in \PP_{2}(\Omega)\mapsto W_2^2(\mu,\nu)+\gamma E(\mu) \]
where $\gamma \geq 0$. If $\mu\in \PP_{2}(\Omega)$ minimizes $J$, then there exists a subgradient $\phi^{\mu,\nu}\in\mathbb{L}_1(\mu)$ of  $W_2^2(\cdot,\nu)$ at $\mu$ and  a potential $\psi\in\mathbb{L}_1(\nu)$ verifying
$
\phi^{\mu,\nu}(x)+\psi(y) \le \vert x-y \vert ^2
$
for all $x,y$ in $\Omega$ such that $(\phi^{\mu,\nu},\psi)$ is an optimal couple of the Kantorovich's dual problem associated to $\mu,\nu$ (Theorem \ref{Kantorovich}). Moreover, for all $\eta\in \PP_{2}(\Omega)$,
\begin{equation}
\label{ineg_subgradient}
 \gamma \ \langle\nabla E(\mu),\mu-\eta\rangle\le - \int \phi^{\mu,\nu} d(\mu-\eta).
\end{equation}
\end{thm}

\section{Existence, uniqueness and stability of penalized barycenters} \label{sec:exist}

In this section, we present some properties of the minimizers of the functional $J_{\P}^{\gamma}$ defined in \eqref{defini_J} in terms of existence, uniqueness and stability.

\subsection{Existence and uniqueness}
We first sconsider  the minimization problem \eqref{def_J}  in the particular setting where $\P$ is a discrete distribution on $\PP_2(\Omega)$. That is, we study the problem
\begin{equation}
\label{J_discret}
\underset{\mu\in\PP_2(\Omega)}{\min}\ J_{\P_n}^{\gamma}(\mu)=\int W_2^2(\mu,\nu)d\P_n(\nu)+\gamma E(\mu)=\frac{1}{n}\sum_{i=1}^nW_2^2(\mu,\nu_i)+\gamma E(\mu)
\end{equation}
where $\P_n=\frac{1}{n}\sum_{i=1}^n\delta_{\nu_i}\in W_2(\PP_2(\Omega))$ with $\nu_1,\ldots,\nu_n$  measures in $\PP_2(\Omega)$.
\begin{thm}
\label{exist_discrete}
Suppose that Assumption \ref{hyp_E} holds and that $\gamma > 0$. Then, the functional $J_{\P_n}^{\gamma}$ defined by \eqref{J_discret} admits a unique minimizer on $\PP_2(\Omega)$ which belongs to the domain $\DD(E)$ of the penalizing function $E$, as defined in \eqref{eq:domE}.
\end{thm}

The proof of  Theorem \ref{exist_discrete} is given in the Appendix \ref{App_exist}. Thanks to this result, one may  impose the penalized Wasserstein barycenter $\mu_{\P_{n}}^{\gamma}$  to be a.c.\ on $\Omega$ by choosing a penalization function $E$ with value $+\infty$ outside of the space of a.c. distributions. %he negative entropy  $E_e$ introduced in \eqref{def:negative_entropy} for the penalization function. 
For this choice,  \eqref{J_discret} becomes a problem of minimization over a set of pdf. % {\color{red} virer ce qui suit? with entropy penalization: oui je suis d'accord sachant qu'on fait deja cette remarque dans la partie 2 pour le cas entropie}  %Examples of the use of the negative entropy as a penalization term are given in Section \ref{sec:num} on numerical experiments.

The existence and uniqueness of \eqref{def_J} can now be shown in a general case. Since any probability measure in $\P \in W_2(\PP_2(\Omega))$ can be approximated by a sequence of finitely supported measures $\P_n$  (see Theorem \ref{thm_ex} in Appendix \ref{App_exist}), we can lean on Theorem \ref{exist_discrete} for the  proof of the following result, which is also detailed in the Appendix \ref{App_exist}.
\begin{thm}
\label{exist_continuous}
Let $\P \in W_2(\PP_2(\Omega))$. Suppose that Assumption \ref{hyp_E} holds and that $\gamma > 0$. Then, the functional $J_{\P}^{\gamma}$ defined by \eqref{def_J} admits a unique minimizer.
\end{thm}

\subsection{Stability}

% In that objective, the Bregman divergence is well adapted, see for instance \cite{censor1997parallel,bregman_web} for an overview of the Bregman divergence in optimization.

When $\gamma > 0$, we now study the stability of the minimizer of $J_{\P}^{\gamma}$ with respect to discrete distributions $\P$ and the symmetric Bregman divergence $d_{E}$ \eqref{eq:dE} associated to a smooth relative $G$-functional $E$. Set $\nu_1,\ldots,\nu_n\in\PP_2(\Omega)$ and $\eta_1,\ldots, \eta_n \in\PP_2(\Omega)$. We denote by $\P_n^{\nu}$ (resp.\ $\P_n^{\eta}$) the discrete measure $\frac{1}{n}\sum_{i=1}^n\delta_{\nu_i}$ (resp.\ $\frac{1}{n}\sum_{i=1}^n\delta_{\eta_i}$)  in $W_2(\PP_2(\Omega))$. %As well, $\eta_1,\ldots, \eta_n$ are measures in $\PP_2(\Omega)$ and $\P_n^{\eta}=\frac{1}{n}\sum_{i=1}^n\delta_{\eta_i}$. 
\begin{thm}
\label{stability}
Let  $E$ be a smooth relative $G$-functional thus satisfying Assumption  \ref{hyp_E} and $\Omega$ a compact subset of $\R^d$. Let $\mu_{\nu},\mu_{\eta}\in\PP_{2}(\Omega)$ with $\mu_{\nu}$ minimizing $J_{\P_n^{\nu}}^{\gamma}$ and $\mu_{\eta}$ minimizing $J_{\P_n^{\eta}}^{\gamma}$ defined by \eqref{J_discret}. Then, the symmetric Bregman divergence associated to $E$ can be upper bounded as follows
\begin{equation}
d_E(\mu_{\nu},\mu_{\eta}) \le \frac{4\diam(\Omega)}{\gamma n}\ \underset{\sigma\in\mathcal{S}_n}{\inf}\sum_{i=1}^nW_2(\nu_i,\eta_{\sigma(i)}), \label{eq:boundperm}
\end{equation}
where $\mathcal{S}_n$ is the permutation group of the set $\{1,\ldots,n\}$, and $\diam(\Omega)$ stands for the diameter of $\Omega$.
\end{thm}
The proof of  Theorem \ref{stability} is given in  Appendix \ref{App_stab}.  
To better interpret the upper bound \eqref{eq:boundperm}, we need the notion of Kantorovich transport distance $\mathcal{T}_{W_2}$ on the metric space $(\PP_2(\Omega),W_2)$, see  \cite{villani2003topics}. %We denote by $Z(M,d)$ the space of all probability measures $\P$ supported on $M$ such that  $\int_M d(\mu,\mu_0)d\P(\mu)< + \infty$,  for some (or equivalently for all) $\mu_0\in M$.
For $\P,\mathbb{Q}\in W_2(\PP_2(\Omega))$ endowed with the Wasserstein distance $W_2$, we have that
\[\mathcal{T}_{W_2}(\P,\mathbb{Q}):=\underset{\Pi}{\inf}\int_{\PP_2(\Omega)\times\PP_2(\Omega)}W_2(\mu,\nu)d\Pi(\mu,\nu),\]
where the minimum is taken over all probability measures $\Pi$ on the product space $\PP_2({\Omega})\times\PP_2(\Omega)$  with marginals $\P$ and $\mathbb{Q}$. Since $\P_n^{\nu}$ and $\P_n^{\eta}$ are discrete probability measures supported on $\PP_2(\Omega)$, it follows that the upper bound \eqref{eq:boundperm} in Theorem \ref{stability} can also be written as (by Birkhoff's theorem for bi-stochastic matrices, see e.g.\ \cite{villani2003topics})
\[d_E(\mu_{\nu},\mu_{\eta})\le \frac{4\diam(\Omega)}{\gamma} \mathcal{T}_{W_2}(\P_n^{\nu},\P_n^{\eta}).\]
Hence the above upper bound means that the Bregman divergence between the penalized Wasserstein barycenters $\mu_{\nu}$ and $\mu_{\eta}$ is controlled by the Kantorovich transport distance between the distributions $\P_n^{\nu}$ and $\P_n^{\eta}$.

\subsection{Discussion on the stability Theorem}\label{subsec:discussion}
Theorem \ref{stability} is of particular interest   in the  setting where the $\nu_i$'s and $\eta_i$'s are discrete probability measures on $\R^{d}$. If we assume  that $\nu_{i}=\frac{1}{p}\sum_{j=1}^{p}\delta_{\bX_{i,j}}$ and $\eta_{i}=\frac{1}{p}\sum_{j=1}^{p}\delta_{\bY_{i,j}}$ where $(\bX_{i,j})_{1\le i\le n;1\le j\le p}$ and $(\bY_{i,j})_{1\le i\le n;1\le j\le p}$ are (possibly random) vectors in $\R^{d}$, then by \eqref{eq:boundperm},
\begin{align*}
d_E(\mu_{\nu},\mu_{\eta}) &\le \frac{4\diam(\Omega)}{\gamma n} \ \underset{\sigma\in\mathcal{S}_n}{\inf} \ \sum_{i=1}^n\left(\underset{\lambda\in\mathcal{S}_{p}}{\inf}\left\{\frac{1}{p}\sum_{j=1}^{p}\vert \bX_{i,j}-\bY_{\sigma(i),\lambda(j)}\vert^2\right\}\right)^{1/2}
\end{align*}
where computing $W_2$ becomes an assignment task through the estimation of permutations $\sigma$ and $\lambda$.

Theorem \ref{stability} is also  useful to compare the 
penalized Wasserstein barycenters respectively obtained  from  data made of $n$ a.c.\ probability measures $\nu_{1},\ldots,\nu_{n}$  and  from their empirical counterpart $\bnu_{p_{i}} = \frac{1}{p_{i}} \sum_{j=1}^{p_{i}} \delta_{\bX_{i,j}}$, where $(\bX_{i,j})_{j=1,\ldots,p_i}$ are iid and generated from $\nu_i$. 
Denoting as $\hat{\bmu}_{n,p}^{\gamma}$  the random density satisfying
$$
\hat{\bmu}_{n,p}^{\gamma} = \underset{\mu\in\PP_{2}(\Omega)}{\text{argmin}}\  \frac{1}{n} \sum_{i=1}^{n} W_2^2\left(\mu,\frac{1}{p_{i}} \sum_{j=1}^{p_{i}} \delta_{\bX_{i,j}} \right) +\gamma E(\mu),
$$
it  follows from  inequality \eqref{eq:boundperm} that
\begin{equation}
\E \left(  d^{2}_E\left(\mu_{\P_n^{\nu}}^{\gamma},\hat{\bmu}_{n,p}^{\gamma}\right)  \right) \le \frac{16\diam(\Omega)}{\gamma^{2} n}   \sum_{i=1}^n \E \left( W_2^2(\nu_i,\bnu_{p_{i}}) \right).\label{eq:discrete2}
\end{equation}
This result allows to discuss the rate of convergence (for the squared symmetric Bregman divergence)  of $\hat{\bmu}_{n,p}^{\gamma}$ to $\mu_{\P_n^{\nu}}^{\gamma}$ as a function of the rate of convergence (for the squared Wasserstein distance) of the empirical measure  $\bnu_{p_{i}}$ to $\nu_{i}$ for each $1 \le i \le n$  (in the asymptotic setting where $p = \min_{1 \le i \le n} p_{i}$ is let going to infinity).

As an illustrative example, in the one-dimensional case  $d=1$ and  for absolutely continuous measures, one may use
%the work in \cite{W1} on a detailed study of the variety of rates of convergence of an empirical measure on the real line toward its population counterpart for the expected squared Wasserstein distance. For example,
 Theorem 5.1 in  \cite{W1}, to obtain that 
$$
\E \left( W_2^2(\nu_i,\bnu_{p_{i}}) \right) \leq \frac{2}{p_{i}+1}K(\nu_i), \mbox{ with } K(\nu_i) = \int_{\Omega} \frac{F_{i}(x)(1-F_{i}(x))}{f_{i}(x)} dx,
$$
where $f_{i}$ is the pdf of $\nu_i$, and $F_{i}$ denotes its cumulative distribution function. Therefore, provided that $K(\nu_{i})$ is finite for each $1 \leq i \leq n$, one obtains the following rate of convergence of $\hat{\bmu}_{n,p}^{\gamma}$ to $\mu_{\P_n^{\nu}}^{\gamma}$ for $d=1$
% (in the case of measures $\nu_{i}$ supported on an interval $\Omega$ of $\R$)
\begin{equation}
\E \left(  d^{2}_E\left(\mu_{\P_n^{\nu}}^{\gamma},\hat{\bmu}_{n,p}^{\gamma}\right)  \right) \le \frac{32\diam(\Omega)}{\gamma^{2} n}   \sum_{i=1}^n \frac{K(\nu_i)}{p_{i}+1} \leq  \frac{32\diam(\Omega)}{\gamma^{2}} \left( \frac{1}{n}   \sum_{i=1}^n  K(\nu_i) \right) (p+1)^{-1}. \label{eq:exstability}
\end{equation}
%Note that by the results in Appendix A in \cite{W1}, a necessary condition for $J_{2}(\nu_{i})$ to be finite is to assume that $f_{i}$ is almost everywhere positive on the interval $\Omega$.
Rate of convergence in $W_2$ distance between a discrete measure and its empirical counterpart are also given in one-dimension in \cite{W1}.
When the measures $\nu_{1},\ldots,\nu_{n}$ are supported on $\R^{d}$ with $d \geq 2$, we refer to the discussion in Section \ref{sec:discuss} which uses the results in \cite{fournier:hal-00915365} on the rate of convergence of an empirical measure in Wasserstein distance to derive rates of convergence for $d_E\left(\mu_{\P_n^{\nu}}^{\gamma},\hat{\bmu}_{n,p}^{\gamma}\right)$.

\section{Convergence properties of penalized empirical barycenters} \label{sec:conv}
In this section, we study for $\Omega\subset \R^d$ compact the convergence of the penalized  Wasserstein barycenter of a set $\bnu_1,\ldots,\bnu_n$ of independent random measures sampled from a distribution $\P$   towards a minimizer of $J_{\P}^0$, i.e.\ a population Wasserstein barycenter of the probability distribution $\P\in W_2(\PP_2(\Omega))$. Throughout this section, it is assumed that $E$ is a smooth relative $G$-functional so that it satisfies Assumption  \ref{hyp_E}. We first introduce and recall some notations.
\begin{defi}
For $\bnu_1,\ldots,\bnu_n$ iid random measures in $\PP_2(\Omega)$ sampled from a distribution $\P \in W_2(\PP_2(\Omega))$, we set $\P_n=\frac{1}{n}\sum_{i=1}^n\delta_{\bnu_i}$. Moreover, we use the notation  (with $\gamma > 0$)
\begin{align}
\label{fPn_gamma}
\bmu_{\P_n}^{\gamma}& = \underset{\mu\in\PP_{2}(\Omega)}{\text{argmin}}\ J_{\P_n}^{\gamma}(\mu)=\int W_2^2(\mu,\nu)d\P_n(\nu)+\gamma E(\mu)\\
\label{fP_gamma}
\mu_{\P}^{\gamma}& = \underset{\mu\in\PP_{2}(\Omega)}{\text{argmin}}\ J_{\P}^{\gamma}(\mu)=\int W_2^2(\mu,\nu)d\P(\nu)+\gamma E(\mu)\\
\label{fP_0}
\mu_{\P}^0&\in\underset{\mu\in\PP_{2}(\Omega)}{\text{argmin}}\ J_{\P}^0(\mu)=\int W_2^2(\mu,\nu)d\P(\nu),
\end{align}
that will be respectively referred as to the penalized empirical Wasserstein barycenter \eqref{fPn_gamma}, the penalized population Wasserstein  barycenter \eqref{fP_gamma} and the population Wasserstein  barycenter \eqref{fP_0}. 
\end{defi}

\begin{rmq} \label{rmq:uniqueness}
Thanks to Theorem \ref{exist_discrete}, one has that the penalized Wasserstein barycenters $\bmu_{\P_n}^{\gamma}$ and $\mu_{\P}^{\gamma}$ are well defined in the sense that they are the unique minimizers of $J_{\P_n}^{\gamma}$ and $J_{\P}^{\gamma}$ respectively. By Theorem 2 in \cite{gouic2015existence}, there exists a population Wasserstein barycenter  $\mu_{\P}^0$ but it is not necessarily unique. Nevertheless, as argued in \cite{AGUEH2017812}, a sufficient condition for the uniqueness of $\mu_{\P}^0$ is to assume that the distribution $\P$ gives a strictly positive mass to the set of a.c.\ measures with respect to the Lebesgue measure. Moreover, if $\P$ is supported on the set of measures $\PP_2(\Omega)\cap \mathbb{L}_q(\Omega)$ for some $q\in(1,+\infty)$ (i.e. $\PP_2(\Omega)$ measures with $\mathbb{L}_q(\Omega)$ densities), it follows that  $\mu_{\P}^0$  admits a density in $\mathbb{L}_q(\Omega)$. This property is also claimed in \cite{AGUEH2017812}, and proved from displacement convexity arguments of the $\mathbb{L}_q(\Omega)$ norm (see \cite{agueh2011barycenters} for the discrete case).
\end{rmq}  

In what follows, we discuss  some convergence results on the  penalized Wasserstein barycenters $\mu_{\P}^{\gamma}$  as $\gamma$ tends to $0$ and  $\bmu_{\P_n}^{\gamma}$ as $n$ tends to $+\infty$. To this end, we will need tools borrowed from the empirical process theory (see \cite{van1996weak}).
\begin{defi} \label{def:covnumber}
Let $\mathcal{F}=\{f:U\mapsto \R\}$ be a class of real-valued functions defined on a given set $U$, endowed with a norm $\| \cdot \|$. An envelope function $F$ of $\mathcal{F}$ is any function $u\mapsto F(u)$ such that $\vert f(u) \vert\le F(u)$ for every $u\in U$ and $f\in\mathcal{F}$. The minimal envelope function is $u\mapsto\sup_f \vert f(u)\vert$. The covering number $N(\epsilon,\mathcal{F},\Vert\cdot\Vert)$ is the minimum number of balls $\{\Vert g-f\Vert< \epsilon\}$ of radius $\epsilon$ and center $g$ needed to cover the set $\mathcal{F}$. The metric entropy is the logarithm of the covering number. Finally, we define
\begin{equation}
\label{def_metric_entropy}
I(\delta,\mathcal{F})=\underset{Q}{\sup}\int_0^{\delta}\sqrt{1+\log N(\epsilon \Vert F\Vert_{\mathbb{L}_2(Q)},\mathcal{F},\Vert \cdot\Vert_{\mathbb{L}_2(Q)})}d\epsilon
\end{equation}
where the supremum is taken over all discrete probability measures $Q$ supported on $U$ with $\Vert F\Vert_{\mathbb{L}_2(Q)}=\left(\int \vert F(u)\vert^2 dQ(u)\right)^{1/2}>0$. The term $I(\delta,\mathcal{F})$ is essentially the integral of the square root of the metric entropy along the radius of the covering balls of $\mathcal{F}$.
\end{defi}
The proof of the following theorems are given in Appendix \ref{App_conv}.

\subsection{Convergence of \texorpdfstring{$\mu_{\P}^{\gamma}$ towards $\mu_{\P}^0$}  \ \  } \label{sec:convbias}
We here present  convergence results  of the penalized population Wasserstein barycenter $\mu_{\P}^{\gamma}$ toward $\mu_{\P}^0$ as $\gamma \to 0$. This is classically referred to as the convergence of the bias term in nonparametric statistic.
\begin{thm}\label{th:biais}
Suppose that $\Omega$ is a compact of $\R^d$. Then, every limit of a subsequence of $(\mu_{\P}^{\gamma})_{\gamma}$  in the metric space $(\PP_{2}(\Omega),W_2)$ is a population Wasserstein barycenter. If we further assume that $\mu_{\P}^0$ is unique, then one has that
$$\lim\limits_{\gamma \rightarrow 0} W_2(\mu_{\P}^{\gamma},\mu_{\P}^0)=0.$$
Moreover, if $\mu_{\P}^0 \in \DD(E)$  and  $\nabla E(\mu_{\P}^0)$ is a continuous function on $\Omega$ then
$$\lim\limits_{\gamma \rightarrow 0} D_E(\mu_{\P}^{\gamma},\mu_{\P}^0)= 0,$$
where $D_E$ is the non-symmetric Bregman divergence defined by
\begin{equation}
D_E(\mu_{\P}^{\gamma},\mu_{\P}^0) = E(\mu_{\P}^{\gamma}) - E(\mu_{\P}^0) - \langle \nabla E(\mu_{\P}^0) , \mu_{\P}^{\gamma} - \mu_{\P}^0 \rangle.  \label{eq:nonsymDE}
\end{equation}
%Moreover, if $\nabla E(\mu_{\P}^0)$ is a continuous function on $\Omega$ then
%$$\lim\limits_{\gamma \rightarrow 0} D_E(\mu_{\P}^{\gamma},\mu_{\P}^0)= 0$$
%where $D_E$ denotes the Bregman divergence between two measures $\mu$ and $\nu$ defined as
%$
%D_E(\mu,\nu)= E(\mu) - E(\nu) -  \langle \nabla E(\nu), \mu-\nu \rangle.
%$

\end{thm}

\subsection{Convergence of \texorpdfstring{$\bmu_{\mathbb{P}_n}^{\gamma}$ towards $\mu_{\mathbb{P}}^{\gamma}$} \ \ }\label{sec:conv_var}

We establish a general result  about the convergence to zero of ${\E}(d_E^2(\bmu_{\P_n}^{\gamma},\mu_{\P}^{\gamma}))$ that is referred to as the variance term. Complementary results on the rate of convergence of this variance term are then given. These additional results are shown to be useful to obtain a data-driven choice for the regularization paper $\gamma$ as detailed in the companion paper \cite{BCP18} where we provide numerical experiments illustrating the use of penalized Wasserstein barycenters for data analysis.

\begin{thm}
\label{th:convergence_var}
 If $\Omega$ is a compact of $\R^d$, then, for any $\gamma >0$, one has that  
\begin{equation}
\lim\limits_{n\rightarrow \infty} {\E}(d_E^2(\bmu_{\P_n}^{\gamma},\mu_{\P}^{\gamma})) = 0
\end{equation}
\end{thm}

%The proof of Theorem \ref{th:convergence_var} is given in the Appendix \ref{App_conv} and depend on the subgradient's inequality.

% from   upper bound the quantity $I(1,\mathcal{H})\Vert H\Vert_{\mathbb{L}_{2}(\P)}$. From Definition \ref{def:covnumber}, the  metric entropy tends to infinity as $\epsilon$ tends to zero. Hence one needs to control the rate of convergence of the metric entropy towards infinity to ensure that $I(1,\mathcal H)<\infty$.

% \subsection{Rate of convergence between \texorpdfstring{$\bmu_{\mathbb{P}_n}^{\gamma}$ and $\mu_{\mathbb{P}}^{\gamma}$ } \  and particular cases}\label{sec:ratedE}

We can actually provide a rate of convergence for this variance term which deeply depends  on compactness properties of the space of measures considered in the minimization problem \eqref{J_discret}. To this end, we introduce the class of functions
$$
\mathcal{H}=\{h_{\mu}:\nu\in\PP_2(\Omega)\mapsto W_2^2(\mu,\nu)\in\R\ \mbox{for} \ \mu\in\PP_2(\Omega)\}.
$$

\begin{thm}
\label{rate_convergence}
If $\Omega$ is a compact of $\R^d$, then one has that
\begin{equation}
\label{rate_conv1}
{\E}(d_E^2(\bmu_{\P_n}^{\gamma},\mu_{\P}^{\gamma}))\le \frac{CI(1,\mathcal{H})\Vert H\Vert_{\mathbb{L}_2(\P)}}{\gamma^2n}
\end{equation}
where $C$ is a positive constant depending on $\Omega$, $H$ is an envelope function of $\mathcal{H}$ and $I(1,\mathcal{H})$ is defined in \eqref{def_metric_entropy}.
\end{thm}
 
To complete this result in a satisfying way, one needs to prove that $I(1,\mathcal H)$ is bounded, which depends on the rate of convergence of the metric entropy towards infinity as the radius $\epsilon$ of the covering balls tends to zero.
  
\paragraph{The one-dimensional case} \label{sec:1Dcase}
Studying the metric entropy of the class $\mathcal{H}$ boils down to studying the metric entropy of the space $(\PP_2(\Omega),W_2)$. By approximating each measure by discrete ones, this corresponds to the metric entropy of the space of discrete distributions on $\Omega$, which is of order $1/\epsilon^{d}$ where $d$ is the dimension of $\Omega$ assumed to be compact (see e.g. \cite{nguyen2011wasserstein}). The term $I(1,\mathcal{H})$ appearing in \eqref{rate_conv1} is thus finite in the one dimensional case.

\begin{thm}
\label{th:rate1D}
If $\Omega$ is a compact subset of $\R$, then there exists a finite constant $C>0$ such that
$${\E}(d_E^2(\bmu_{\P_n}^{\gamma},\mu_{\P}^{\gamma}))\le \frac{C}{\gamma^2n}.$$
\end{thm}

%Therefore, when $\bnu_{1},\ldots,\bnu_{n}$ are iid random measures with support included in a compact interval $\Omega$, it follows  from Theorems \ref{th:biais} and \ref{th:rate1D} that if $\gamma = \gamma_{n}$ is such that $\lim_{n \to \infty} \gamma_{n}^2n = + \infty$ then
%$
%\lim_{n \to \infty} \E(d_E^2\left(\bmu_{\P_n^{\bnu}}^{\gamma},\mu_{\P}^{0})\right) = 0.
%$

\paragraph{The $d$-dimensional case with additional penalization.} \label{sec:genDcase}

In the case $d \ge 2$, the class of functions $\mathcal{H}=\{h_{\mu}: \mu\in\PP_2(\Omega)\}$ is too large to control the metric entropy and to  have a finite value for $I(1,\mathcal H)$. To tackle this issue, we impose more smoothness on the penalized Wasserstein barycenter.
More precisely, we  assume that $\Omega$ is a smooth and uniformly convex set, and  for a smooth relative $G$-functional  with reference measure $\lambda = dx$ (that we denote by $E_{G}$) we choose  the penalizing function
% that and $(\nu_i)_{i=1,\ldots,n}$ are of law $\P$.
\begin{equation}
\label{Sobolev_penalty}
E(\mu) = \left\{\begin{array}{ll}
 E_{G}(\mu) + \Vert f\Vert_{H^{k}(\Omega)}^{2} = \int_{\Omega} G(f(x)) dx+\Vert f\Vert_{H^{k}(\Omega)}^{2}, & \mbox{if}\ f =\frac{d\mu}{dx} \ \mbox{and} \ f\geq\alpha, \\
+\infty & \mbox{otherwise.}
\end{array}\right.
\end{equation}
where $\Vert\cdot\Vert_{H^k(\Omega)}$ denotes the Sobolev norm associated to the $\mathbb{L}^2(\Omega)$ space and $\alpha>0$ is arbitrarily small. Remark that we could choose a linear combination with different weights for the relative $G$-functional and the squared Sobolev norm. Then, the following result holds.
\begin{thm} \label{th:rateSobolev}
Suppose that $\Omega$ is a compact and uniformly convex set with a $C^1$ boundary. Assume that the penalty function $E$ is given by \eqref{Sobolev_penalty} for some $\alpha > 0$ and $k > d-1$.  Then, there exists a finite contant $C > 0$ such that
$$
\E\left(d_{E_G}^2\left(\bmu_{\P_n}^{\gamma},\mu_{\P}^{\gamma}\right)\right) \le \E\left(d_{E}^2\left(\bmu_{\P_n}^{\gamma},\mu_{\P}^{\gamma}\right)\right) \le \frac{C}{\gamma^2n}. % \; \mbox{where}\ E_e \ \mbox{is the negative entropy given in \eqref{def:negative_entropy}}.
$$
\end{thm}
\section{Discussion} \label{sec:discuss}
In order to sum up the different results, let us consider the setting of Subsection \ref{subsec:discussion}. More precisely,  assume that $\P \in W_2(\PP_2(\Omega))$ is a distribution which gives  mass one to the set of a.c.\ measures (with respect to the Lebesgue measure $dx$).  Hence, thanks to Remark \ref{rmq:uniqueness},  there exists a unique  population Wasserstein barycenter  $\mu_{\P}^0$ which is an a.c.\ measure. We also assume that $\Omega$ is a compact and uniformly convex set with a $C^1$ boundary. 

Now, let $\bnu_1,\ldots,\bnu_n$ be iid random measures sampled from $\P$. For each $1 \leq i \leq n$, suppose that, conditionally on $\bnu_{i}$,   we are given a sequence $\bX_{i,1},\ldots,\bX_{i,p_{i}}$ of iid observations in $\R^{d}$ sampled from the measure $\bnu_{i}$.  Recall the notation $\P_n^{\nu}=\frac{1}{n}\sum_{i=1}^n\delta_{\bnu_i}$, $\bnu_{p_{i}} = \frac{1}{p_{i}} \sum_{j=1}^{p_{i}} \delta_{\bX_{i,j}}$ and $p = \min_{1 \le i \le n} p_{i}$, and consider the regularized Wasserstein barycenter $\hat{\bmu}_{n,p}^{\gamma}$ defined as
$$
\hat{\bmu}_{n,p}^{\gamma} = \underset{\mu\in\PP_{2}(\Omega)}{\text{argmin}}\  \frac{1}{n} \sum_{i=1}^{n} W_2^2\left(\mu,\bnu_{p_{i}} \right) +\gamma E(\mu), 
$$
with  the following penalizing functional defined for any arbitrarily small  $\alpha > 0$ 
\begin{equation}
\label{ex_penalty}
E(\mu) = \left\{\begin{array}{ll}
 \frac{1}{2}  \int_{\Omega} |f(x)|^2 dx+\Vert f\Vert_{H^{k}(\Omega)}^{2}, & \mbox{if}\ f =\frac{d\mu}{dx} \ \mbox{and} \ f\geq \alpha, \\
+\infty & \mbox{otherwise.}
\end{array}\right.
\end{equation}

For two a.c.\  measures $\mu = \mu_{f}$ and $\nu = \nu_{g}$ with density $f$ and $g$, it is easily seen that the symmetric and non-symmetric Bregman divergences related to $E$ satisfy
$$
d_E(\mu_{f},\nu_{g}) \geq    \| f - g\|^2_{\mathbb{L}^2(\Omega)}  \quad \mbox{and} \quad D_E(\mu_{f},\nu_{g}) \geq  \frac{1}{2}  \| f - g\|^2_{\mathbb{L}^2(\Omega)}.
$$
Let us now discuss the convergence of the measure $\hat{\bmu}_{n,p}^{\gamma}$ towards $\mu_{\P}^0$ with respect to the squared $\mathbb{L}^2(\Omega)$ distance between their respective densities $\hat{\bfun}_{n,p}^{\gamma}$ and $f_{\P}^0$, when both $n$ and $p$ tend to infinity and $\gamma$ tends to $0$. To this end, it is necessary to assume that $f_{\P}^0 \geq \alpha$, and we consider the decomposition
%$$
%\E \left(   \|\hat{\bfun}_{n,p}^{\gamma} -  f_{\P}^0\|^{2}_{\mathbb{L}^2(\Omega)}  \right) \leq 3\sqrt{ \underbrace{ \E\left(d_{E}^{2} \left(\hat{\bmu}_{n,p}^{\gamma},\mu_{\P_n^{\nu}}^{\gamma}\right)\right)}_{\mbox{Stability term}} }+3 \sqrt{ \underbrace{ \E\left(d_{E}^{2} \left(\mu_{\P_n^{\nu}}^{\gamma},\mu_{\P}^{\gamma}\right)\right)}_{\mbox{Variance term}}}+6  \underbrace{D_{E}\left(\mu_{\P}^{\gamma},\mu_{\P}^{0}\right)}_{\mbox{Bias term}}.
%$$
$$
\E\ \|\hat{\bfun}_{n,p}^{\gamma} -  f_{\P}^0\|^{2}_{\mathbb{L}^2(\Omega)} \leq 3\sqrt{ \underbrace{ \E\ d_{E}^{2} \left(\hat{\bmu}_{n,p}^{\gamma},\mu_{\P_n^{\nu}}^{\gamma}\right)}_{\mbox{Stability term}} }+3 \sqrt{ \underbrace{ \E\ d_{E}^{2} \left(\mu_{\P_n^{\nu}}^{\gamma},\mu_{\P}^{\gamma}\right)}_{\mbox{Variance term}}}+6  \underbrace{D_{E}\left( \mu_{\P}^{\gamma},\mu_{\P}^{0}\right)}_{\mbox{Bias term}}.
$$
Then, we gather the results on the stability Theorem \ref{stability} combined with Theorem 1 in \cite{fournier:hal-00915365},  Theorem \ref{rate_convergence} (convergence of the variance term) and Theorem \ref{th:biais} (convergence of the bias term)  to prove the convergence to zero of the three terms in the right-hand side of the above inequality.
\paragraph{Stability term} Recall that by inequality \eqref{eq:discrete2} one has that
$$
\E \left(  d^{2}_E \left(\hat{\bmu}_{n,p}^{\gamma},\mu_{\P_n^{\nu}}^{\gamma}\right)\right) \le \frac{16\diam(\Omega)}{\gamma^{2} n}   \sum_{i=1}^n \E \left( W_2^2(\bnu_i,\bnu_{p_{i}}) \right).
$$
For each $1 \leq i \leq n$ and conditionally on $\bnu_i$, the convergence to zero of $\E \left( W_2^2(\bnu_i,\bnu_{p_{i}}) \right)$ as $p_{i}$ tends to infinity can be controlled using the results in \cite{fournier:hal-00915365}. For instance, if the measure $\bnu_i$ has a moment of order $q > 4$ then, by  Theorem 1  in \cite{fournier:hal-00915365}, it follows that there exists a constant $C_{q,d} > 0$ (depending only on $q$ and $d$) such that
$$
 \E \left( W_2^2(\bnu_i,\bnu_{p_{i}}) \right) \leq C_{q,d} \E\left( M_q^{2/q}(\bnu_i) \right)  p_{i}^{-1/2}
$$
provided that $d<4$, and where $M_q(\bnu_i)=\int_{\Omega}\vert x\vert^q d\bnu_i(x)$. Hence, under such assumptions on $q$ and $d$, it follows that
\begin{equation}
\E \left(  d^{2}_E \left(\hat{\bmu}_{n,p}^{\gamma},\mu_{\P_n^{\nu}}^{\gamma}\right)\right) \le 4 C_{q,d}    \E\left( M_q^{2/q}({\bnu_1}) \right) \frac{1}{\gamma^{2} p^{1/2}}. \label{eq:convstability}
\end{equation}
Finer results on the convergence rate of $\E \left( W_2^2(\bnu_i,\bnu_{p_{i}}) \right)$ in dimension $2$ can also be found in \cite{ambrosio2018pde} and \cite{ambrosio2018finer}.

\paragraph{Variance term} By Theorem \ref{rate_convergence} , one obtains that
$
 \E\left(d_{E}^{2} \left(\mu_{\P_n^{\nu}}^{\gamma},\mu_{\P}^{\gamma}\right)\right) \le \frac{C}{\gamma^2n}. $

\paragraph{Bias term} By Theorem \ref{th:biais}, $\lim\limits_{\gamma \rightarrow 0} D_E(\mu_{\P}^{\gamma},\mu_{\P}^0)= 0$.

Let us finally assume that the distribution $\P$ is such that  $\E\left( M_q^{2/q}(\bnu_1) \right) < + \infty$. Therefore,  under the various assumptions made in this discussion, and by combining the above results, the expected squared $\mathbb{L}^2(\Omega)$ error $\E \left(   \|\hat{\bfun}_{n,p}^{\gamma} -  f_{\P}^0\|^{2}_{\mathbb{L}^2(\Omega)}  \right)$ converges to zero provided that  $\gamma=\gamma_{n,p}$ is a sequence of regularizing parameters converging to zero such that
$$
\lim_{\min(n,p) \to \infty} \gamma_{n,p}^2n = + \infty \qquad \mbox{and} \qquad \lim_{\min(n,p) \to \infty} \gamma_{n,p}^2 p ^{1/2} = + \infty.
$$

\appendix

% \section{Additional definitions}\label{App_def}

\section{Proofs of the paper}\label{App_proof}

\subsection{Proof of the subgradient's inequality}\label{App_subgrad}
The proof of  Theorem \ref{thm_subgradient} is based on the two succeeding lemmas.
\begin{lemma}
\label{lemma_1}
The two following assertions are equivalent:
\begin{enumerate}
\item $\mu\in  \PP_{2}(\Omega)$ minimizes $J$ over $ \PP_{2}(\Omega)$,
\item there exists a subgradient $\phi\in\partial J(\mu)$ such that $\langle \phi,\eta-\mu \rangle\ge 0$ for all $\eta\in  \PP_{2}(\Omega)$.
\end{enumerate}
\end{lemma}
\begin{proof}[Proof of Lemma \ref{lemma_1}]
\textit{2}$\Rightarrow$\textit{1}. Let $\phi\in\partial J(\mu)$ such that $\langle \phi,\eta-\mu \rangle\ge 0$ for all $\eta\in  \PP_{2}(\Omega)$. By definition of the subgradient, $\forall \ \eta\in  \PP_{2}(\Omega)$, we have $J(\eta)\ge J(\mu) + \langle \phi, \eta-\mu \rangle$ which is greater than $J(\mu)$ by assertion. Hence $\mu$ minimizes $J$.\\
\textit{1}$\Rightarrow$\textit{2}. Take $\mu\in$ int(dom $J$) (that is $J(\mu)<+\infty$) such that $\mu$ is a minimum of $J$ over $ \PP_{2}(\Omega)$. Then the directional derivative of $J$ at the point $\mu$ along $(\eta-\mu)$ exists (Proposition 2.22 in \cite{Clarke}) and satisfies
\begin{equation}
\label{direct_derivative}
J'(\mu;\eta-\mu):=\underset{\underset{t>0}{t\to 0}}{\lim} \frac{J(\mu+t(\eta-\mu))-J(\mu)}{t}\ge 0.
\end{equation}
Remark that $ \PP_{2}(\Omega)$ is a convex set. By Proposition 4.3 of \cite{Clarke}, since $J$ is a proper convex function and $\mu\in$ dom($J$), we obtain the equivalence
\[\phi \in\partial J(\mu) \ \Leftrightarrow \ \langle \phi, \Delta\rangle\le J'(\mu;\Delta) \ \text{for all} \ \Delta\in  \PP_{2}(\Omega).\]
Moreover, since $J$ is proper convex and lower semi-continuous, so is $J'(f;\cdot)$. Given that $ \PP_{2}(\Omega)$ is a Hausdorff convex space, we get by Theorem 7.6 of \cite{aliprantis2006infinite}, that for all $(\eta-\mu)\in \PP_{2}(\Omega)$, $J'(\mu;\eta-\mu)=\sup\{\langle \phi, \eta-\mu \rangle$ where  $\phi$ is such that $\langle \phi, \Delta \rangle \le J'(\mu;\Delta), \forall \Delta\ \text{in} \  \PP_{2}(\Omega)\}.$
Hence by \eqref{direct_derivative} we get $\underset{\phi\in\partial J(\mu)}{\sup}\langle\phi,\eta-\mu\rangle\ge 0$. We then define the ball $B_\epsilon=\{\eta+\mu\in \MM(\Omega) \ \text{such that}\ \Vert \eta\Vert_{TV} \le \epsilon\}$, where $\Vert \cdot\Vert_{TV}$ is the norm of total variation. We still have
\[\underset{\eta\in B_\epsilon\cap \PP_{2}(\Omega)}{\inf} \ \underset{\phi\in\partial J(\mu)}{\sup}\langle \phi,\eta-\mu\rangle \ge 0.\]
Note that $\partial J(\mu)$ in a convex set. Moreover $B_\epsilon\cap  \PP_{2}(\Omega)$ is compact, and $(\phi,\eta)\mapsto\langle \phi,\eta-\mu\rangle$ is bilinear. Thus we can switch the infimum and the supremum by the Ky Fan's theorem (4.36 in \cite{Clarke}). In that way, there exists $\phi\in\partial J(f)$ such that $\underset{\eta\in B_\epsilon\cap \PP_{2}(\Omega)}{\inf} \langle \phi,\eta-\mu\rangle \ge 0$. By convexity of $ \PP_{2}(\Omega)$, any $\zeta\in  \PP_{2}(\Omega)$ can be written as $t(\eta-\mu)+\mu$ for some $t\ge 0$ and $\eta\in B_\epsilon\cap  \PP_{2}(\Omega)$. This concludes the proof of the lemma.
\end{proof}

We also need the following lemma which follows from Proposition 7.17 in \cite{santambrogio2015optimal}.
\begin{lemma}
\label{lemma_2}
Let $\mu\in \PP_{2}(\Omega)$ and $\phi\in\mathbb{L}_1(\mu)$, then
\[\phi\in\partial_1 W_2^2(\mu,\nu) \Leftrightarrow \exists \ \psi\in\mathbb{L}_1(\nu) \ \mbox{such that} \ \phi(x)+\psi(y)\le\vert x-y\vert^2\] and  $W_2^2(\mu,\nu)=\int \phi d\mu + \int\psi d\nu$ 
where $\partial_1W_2^2(\mu,\nu)$ denote the subdifferential of the function $ W_2^2(\cdot,\nu)$ at $\mu$.
\end{lemma}

From these lemmas, we directly get the proof of Theorem \ref{thm_subgradient}.
\begin{proof}[Proof of Theorem \ref{thm_subgradient}]
Let $\mu\in \PP_{2}(\Omega)$ be a minimizer of $J$. From Lemma \ref{lemma_1}, we know that there exists $\phi$ a subgradient of $J$ in $\mu$ such that $\langle \phi,\eta-\mu\rangle\ge 0$ for all $\eta\in  \PP_{2}(\Omega)$. Since $\zeta\mapsto E(\zeta)$ is convex differentiable, $\zeta\mapsto W_2^2(\zeta,\nu)$ is a continuous convex function and $\mu$ minimizes $J$, we have by the subdifferential of the sum (Theorem 4.10 in \cite{Clarke}) that $\partial J(\mu)=\partial_1 W_2^2(\mu,\nu)+\gamma \ \nabla E(\mu)$. This implies that all $\phi\in\partial J(\mu)$ is written $\phi=\phi_1+\phi_2$ with $\phi_1=\phi^{\mu,\nu}$ optimal for the couple $(\mu,\nu)$ (by Lemma \ref{lemma_2}) and $\phi_2=\gamma\nabla E(\mu)$. Finally, we have that $\langle \phi^{\mu,\nu}+\gamma\nabla E(\mu),\eta-\mu\rangle\ge 0$ for all $\eta\in \PP_{2}(\Omega)$ that is
$\gamma \ \langle\nabla E(\mu),\mu-\eta\rangle\le - \int \phi^{\mu,\nu} d(\mu-\eta), \\ \forall \eta\in  \PP_{2}(\Omega).$
\end{proof}

\subsection{Proof of existence, uniqueness and stability of penalized barycenters}\label{App_exist}
For the sake of completeness, we introduce the  functional space $Y:= \{g \in\mathcal{C}(\Omega): x \mapsto g(x)/(1+\vert x\vert ^2)\ \text{is bounded}\}$ 
endowed with the norm \linebreak $\Vert g \Vert_Y=\sup_{x\in\Omega} \vert g(x)\vert/(1+\vert x\vert^2)$ 
where $\mathcal{C}(\Omega)$ is the space of continuous functions from $\Omega$ to $\R$. We finally denote as $Z$ the closed subspace of $Y$ given by
$Z= \left\{g \in\mathcal{C}(\Omega) : \lim_{\vert x\vert\rightarrow\infty} g(x)/(1+\vert x\vert ^2)=0\right\}.$ 
The space $\MM(\Omega)$ of bounded Radon measures is identified with the dual of $\mathcal{C}_0(\Omega)$ (space of continuous functions that vanish at infinity). Finally, we denote by
$
\mathbb{L}_1(\mu)
$
the set of integrable functions $g : \Omega \to \R$ with respect to the measure $\mu$.

\begin{proof}[Proof of Theorem \ref{exist_discrete}]
Let $(\mu^k)_k\subset\PP_2(\Omega)$ a minimizing sequence of probability measures of $J_{\P_n}^{\gamma}$. Hence, there exists a constant $M\ge 0$ such that $\forall k,\ J_{\P_n}^{\gamma}(\mu^k)\le M$. It follows that  for all $k$, $\frac{1}{n}\sum_{i=1}^nW_2^2(\mu^k,\nu_i)\le M$. By Lemma 2.1 of \cite{agueh2011barycenters} we thus have
\begin{align*}
\frac{1}{n}\sum_{i=1}^nW_2^2(\nu^i,\mu^k) &= 2 \sum_{i=1}^n\underset{f\in Z}{\sup}\ \left\{\int_{\Omega}fd\mu^k+\int_{\Omega}Sf(x)d\nu^i(x)\right\} \le M,
\end{align*}
where $S f(x)=\underset{y\in\Omega}{\inf}\{ \frac{1}{2n}\vert x-y\vert ^2-f(y)\}$. Since the function $x\mapsto\vert x\vert^{\alpha}$ (with $1<\alpha<2$) belongs to $Z$, we have that $\int_{\R^d}\vert x\vert ^{\alpha} d\mu^k(x)$ is bounded by a constant $L\ge 0$ for all $k$. We deduce that $(\mu^k)_k$ is tight (for instance, take the compact $K^c=\{x\in\Omega \ \text{such that} \ \vert x\vert ^{\alpha} > \frac{L}{\epsilon}\}$). Since $(\mu^k)_k$ is tight, by Prokhorov's theorem, there exists a subsequence of $(\mu^k)_k$ (still denoted $(\mu^k)_k$) which weakly converges  to a probability measure $\mu$. Moreover, one can prove that $\mu\in \PP_2(\Omega)$. Indeed for all lower semicontinuous functions bounded from below by $f$, we have that $\underset{k\rightarrow \infty}{\liminf}\int_{\Omega}f(x)d\mu^k(x)\ge\int_{\Omega}f(x)d\mu(x)$ by weak convergence. Hence for $f:x\mapsto\vert x\vert ^2$, we get $\int_{\Omega}\vert x\vert ^2d\mu(x)\le\underset{k\rightarrow \infty}{\liminf}\int_{\Omega}\vert x\vert ^2d\mu^k(x) < + \infty$, and thus $\mu\in\PP_2(\Omega)$.

Let $(\pi_i^k)_{1\le i\le n, 1\le k}$ be a sequence of optimal transport plans where $\pi_i^k$ is an optimal transport plan between $\mu^k$ and $\nu_i$. Since $\sup_k W_2^2(\mu^k,\nu_i)=\sup_k \iint_{\Omega\times\Omega}\vert x-y\vert ^2 d\pi_i^k(x,y) < + \infty$, we may apply Proposition 7.1.3 of \cite{ambrosio2008gradient}: $(\pi_i^k)_k$ is weakly relatively compact on the probability space over $\Omega\times\Omega$ and every weak limit $\pi_i$ is an optimal transport plan between  $\mu$ and $\nu_i$ with, for all $1\le i \le n$,
$
W_2^2(\mu,\nu_i)\le \underset{k\rightarrow\infty}{\liminf}\int_{\Omega\times\Omega}\vert x-y\vert^2d\pi_i^k(x,y)< +\infty.
$
Since $E$ is lower semicontinuous, we get that
\begin{align*}
\underset{k\rightarrow\infty}{\liminf} \ J_{\P_n}^{\gamma}(\mu^k) &= \underset{k\rightarrow\infty}{\liminf} \ \frac{1}{n}\sum_{i=1}^nW_2^2(\mu^k,\nu_i)+\gamma E(\mu^k)\\
%&= \frac{1}{n}\sum_{i=1}^n (\underset{k\rightarrow\infty}{\liminf} \ W_2^2(\mu^k,\nu^i)+\gamma E(\mu^k))\\
& \ge \frac{1}{n}\sum_{i=1}^nW_2^2(\mu,\nu_i)+\gamma E(\mu) =J_{\P_n}^{\gamma}(\mu).
\end{align*}
Hence $J_{\P_n}^{\gamma}$ admits at least  $\mu\in\PP_2(\Omega)$ as a minimizer. Finally, by the strict convexity of $J_{\P_n}^{\gamma}$ on its domain, the minimizer is unique and it belongs to $\DD(E)$ as defined in \eqref{eq:domE}, which completes the proof.
\end{proof}

\begin{proof}[Proof of Theorem \ref{exist_continuous}]
First, let us prove the existence of a minimizer. For that purpose, we decide to follow the sketch of the proof of the existence of a Wasserstein barycenter given by Theorem 1 in \cite{gouic2015existence}. We suppose that $(\P_n)_{n\ge 0}\subseteq W_2(\PP_2(\Omega))$ is a sequence of measures, such that  $\mu^n\in \PP_2(\Omega)$ is a probability measure minimizing $J_{\P_n}^{\gamma}$, for all $n$. Furthermore, we suppose that there exists $\P\in W_2(\PP_2(\Omega))$ such that $\mathcal{W}_2(\P,\P_n)\underset{n\rightarrow +\infty}{\longrightarrow}0$. We then have to prove that $(\mu^n)_{n\ge 1}$ is precompact and that all limits minimize $J_{\P}^{\gamma}$. We denote
$\tilde{\bmu}$ a random measure with distribution $\P$  and 
$ \tilde{\bmu}^n$ a random measure with distribution $\P_n$.

\noindent
Hence we get
\begin{align*}
W_2(\mu^n,\delta_x) & =\mathcal{W}_2(\delta_{\mu^n},\delta_{\delta_x})\le \mathcal{W}_2(\delta_{\mu^n},\P_n)+\mathcal{W}_2(\P_n,\delta_{\delta_x})\\
& = \E(W_2^2(\mu^n,\tilde{\bmu}^n))^{1/2}+\E(W_2^2(\tilde{\bmu}^n,\delta_x))^{1/2}.
\end{align*}
Moreover, $\E(W_2^2(\mu^n,\tilde{\bmu}^n))^{1/2}\le M$ for a constant $M\ge 0$ since $\mu_n$ minimizes $J_{\P_n}^{\gamma}$ and $\tilde{\bmu}^n$ is of law $\P_n$. Then for $x\in\Omega$
\begin{align*}
W_2(\mu^n,\delta_x) & \le M + \mathcal{W}_2(\P_n,\delta_{\delta_x}) \le M + \mathcal{W}_2(\P_n,\P) + \mathcal{W}_2(\P,\delta_{\delta_x})\le L
\end{align*}
since $\mathcal{W}_2(\P_n,\P)\underset{n\rightarrow +\infty}{\longrightarrow} 0$ and $\P\in W_2(\PP_2(\Omega))$ by hypothesis. By Markov inequality, we have for $r> 0$
%\begin{align*}\mu^n(B(x,r)^c)=&\P_{\mu^n}(X\in B(x,r)^c)=\P_{\mu^n}(\vert X-x\vert^2\le r^2)\\\le& \frac{\E_{\mu^n}(\vert X-x\vert^2)}{r^2}=\frac{W_2^2(\mu^n,\delta_x)}{r^2}
%\end{align*}
$$
\mu^n(B(x,r)^c)= \P_{\mu^n}(\vert X-x\vert^2\ge r^2) \le  \frac{\E_{\mu^n}(\vert X-x\vert^2)}{r^2}=\frac{W_2^2(\mu^n,\delta_x)}{r^2}
$$
and $\mu^n(B(x,r)^c)\le\frac{L^2}{r^2}$. Hence $(\mu^n)_n$ is tight: it is possible to extract a subsequence (still denoted $(\mu^n)$) which converges weakly to a measure $\mu$ by Prokhorov's theorem. Let us show that $\mu$ minimizes $J_{\P}^{\gamma}$. Let $\eta\in \PP_2(\Omega)$ and $\nu\in\PP_2(\Omega)$ with distribution $\P$.
\begin{align}
J_{\P}^{\gamma}(\eta) & = \E_{\P}(W_2^2(\eta,\nu))+\gamma E(\eta) \nonumber  \\
&= \mathcal{W}_2^2(\delta_{\eta},\P)+\gamma E(\eta) \nonumber  \\
&= \underset{n\rightarrow +\infty}{\lim} \mathcal{W}_2^2(\delta_{\eta},\P_n)+\gamma E(\eta)\ \qquad &\mbox{since by hypothesis}\ \mathcal{W}_2(\P_n,\P)\rightarrow 0   \nonumber  \\
& \ge \underset{n\rightarrow +\infty}{\liminf} \ \mathcal{W}_2^2(\delta_{\mu^n},\P_n)+\gamma E(\mu^n) \qquad &\mbox{since}\ \mu^n \ \mbox{minimizes} \ J_{\P_n}^{\gamma} \label{ineq:exists}
\end{align}
Moreover, we have by the inverse triangle inequality that
$$\underset{n\rightarrow +\infty}{\liminf} \ \mathcal{W}_2(\delta_{\mu^n},\P_n) \geq \underset{n\rightarrow +\infty}{\liminf} \ (\mathcal{W}_2(\delta_{\mu^n},\P) - \mathcal{W}_2(\P,\P_n)).$$
First, $\mathcal{W}_2(\P,\P_n)\rightarrow 0$ by assumption. Second, we have that
\begin{align*}
\underset{n\rightarrow +\infty}{\liminf}\ \mathcal{W}_2(\delta_{\mu^n},\P) &\geq \int \underset{n\rightarrow +\infty}{\liminf}\ W_2^2(\mu_n,\nu)d\P(\nu) &\qquad \mbox{by Fatou's Lemma}&\\
& \geq  \int W_2^2(\mu,\nu)d\P(\nu)=\mathcal{W}_2^2(\delta_{\mu},\P)\qquad&  &
\end{align*}
Thus from \eqref{ineq:exists} and by lower semicontinuity of $E$, we conclude that
$J_{\P}^{\gamma}(\eta)\ge \linebreak \mathcal{W}_2^2(\delta_{\mu},\P) +\gamma E(\mu)=J_{\P}^{\gamma}(\mu).$
Hence $\mu$ minimizes $J_{\P}^{\gamma}$. To finish the proof of the existence of a minimizer, we need the following result whose proof can be found in \cite{gouic2015existence}.

\begin{thm}
\label{thm_ex}
For all $\P\in W_2(\PP_2(\Omega))$, there is a sequence of finitely supported distributions $\P_n$ (that is $\P_n=\sum_{k=1}^K\lambda_k\delta_{\kappa_k}$ where $\sum_{k=1}^K\lambda_k=1$) such that \linebreak $\mathcal{W}_2^2(\P_n,\P)\underset{n\rightarrow +\infty}{\longrightarrow} 0$.
\end{thm}

Now, by Theorem \ref{thm_ex} it follows that for a given distribution $\P$, one can find a sequence of finitely supported distributions $\P_n$ such that for all $n$ there exists a unique measure $\mu^n\in \PP_2(\Omega)$ minimizing $J_{\P_n}^{\gamma}$ using Theorem \ref{exist_discrete} and such that $W_2^2(\P_n,\P)\underset{n\rightarrow +\infty}{\longrightarrow} 0$ thanks to Theorem \ref{thm_ex}. Therefore there is a probability measure $\mu$ which minimizes $J_{\P}^{\gamma}$. Let us make sure that $\mu$ is indeed in the space $\PP_2(\Omega)$. From Theorem \ref{exist_discrete}, we also have that $\mu^n\in \PP_2(\Omega)$ for all $n$. Thus by weak convergence,
$
\int_{\Omega}\vert x \vert^2d\mu(x)\le\underset{n\rightarrow +\infty}{\liminf}\int_{\Omega}\vert x \vert^2d\mu^n(x) < +\infty.
$
Finally, the uniqueness of the minimum is obtained by the strict convexity of the functional $\mu\mapsto\E_{\P}(W_2^2(\mu,\nu))+\gamma E(\mu)$ on the domain $\DD(E)$, which completes the proof.
\end{proof}
\subsection{Proof of the stability Theorem \ref{stability}}\label{App_stab}
\begin{proof}[Proof of Theorem \ref{stability}]
We denote by $\mu,\zeta\in\PP_{2}(\Omega)$ the probability measures such that $\mu$ minimizes $J_{\P_n^{\nu}}^{\gamma}$ and $\zeta$ minimizes $J_{\P_n^{\eta}}^{\gamma}$. For each $1 \leq i \leq n$, one has that $\theta\mapsto\frac{1}{n}W_2^2(\theta,\nu_i)$ is a convex, proper and continuous function. Therefore, Theorem 4.10 in \cite{Clarke}, we have that $\partial J_{\P_n^{\nu}}(\mu)=\frac{1}{n}\sum_{i=1}^n\partial_1 W_2^2(\mu,\nu_i)+\gamma \nabla E(\mu)$. Hence by Lemma \ref{lemma_2}, any $\phi\in\partial J_{\P_n^{\nu}}(\mu)$ is of the form $\phi=\frac{1}{n}\sum_{i=1}^n\phi_i+\gamma\ \nabla E(\mu)$ where for all $i=1,\ldots,n$, $\phi_i=\phi^{\mu,\nu_i}$  is optimal in the sense that $(\phi^{\mu,\nu_i},\psi^{\mu,\nu_i})$ is an optimal couple associated to $(\mu,\nu_i)$ in the Kantorovich formulation of the Wasserstein distance (see Theorem \ref{Kantorovich}). 
Therefore by Lemma \ref{lemma_1}, there exists $\phi=\frac{1}{n}\sum_{i=1}^n\phi^{\mu,\nu_i}+\gamma \nabla E(\mu)$ such that $\langle \phi,\theta-\mu\rangle\ge 0$ for all $\theta\in \PP_{2}(\Omega)$. Likewise, there exists $\breve{\phi}=\frac{1}{n}\sum_{i=1}^n\phi^{\zeta,\eta_i}+\gamma \nabla E(\zeta)$ such that $\langle \breve{\phi},\theta-\zeta\rangle\ge 0$ for all $\theta\in \PP_{2}(\Omega)$. Finally, we obtain 
$$\gamma \langle \nabla E(\mu)-\nabla E(\zeta),\mu-\zeta\rangle \le -\int_{\Omega}\left(\frac{1}{n}\sum_{i=1}^n(\phi^{\mu,\nu_i}-\phi^{\zeta,\eta_i})\right)d(\mu-\zeta).$$
Following the proof of Kantorovich duality's theorem in \cite{villani2003topics}, we can restrict the supremum over $(\phi,\psi)\in C_W$ in Kantorovich's duality Theorem \ref{Kantorovich} to the admissible pairs $(\phi^{cc},\phi^c)$ where
$
\phi^c(y)={\inf}_x\{\vert x-y\vert^2-\phi(x)\}\
$
and
$
\phi^{cc}(x)={\inf}_y\{\vert x -y\vert^2-\phi^c(y)\}.
$
Then, we replace  $\phi^{\mu,\nu_i}$ by $(\phi^{\mu,\nu_i})^{cc}$ (resp. $\phi^{\zeta,\eta_i}$ by $(\phi^{\zeta,\eta_i})^{cc}$ ) and $\psi^{\mu,\nu_i}$ by $(\phi^{\mu,\nu_i})^c$  (resp. $\psi^{\zeta,\eta_i}$ by $(\phi^{\zeta,\eta_i})^c$ ) 
and  obtain 
\begin{align*}
\gamma\langle \nabla E(\mu)-&\nabla E(\zeta),\mu-\zeta\rangle \le -\frac{1}{n}\sum_{i=1}^n\int_{\Omega} \left[(\phi^{\mu,\nu_i})^{cc}(x)-(\phi^{\zeta,\eta_i})^{cc}(x)\right]d(\mu-\zeta)(x)\\
&=-\frac{1}{n}\sum_{i=1}^n  \iint_{\Omega\times\Omega}\left[(\phi^{\mu,\nu_i})^{cc}(x)-(\phi^{\zeta,\eta_i})^{cc}(x)\right]d(\pi^{\mu,\nu_i}-\pi^{\zeta,\eta_i})(x,y),
\end{align*}
where $\pi^{\mu,\nu_i}$ is an optimal transport plan on $\Omega\times\Omega$ with marginals $\mu$ and $\nu_i$ for $i\in\{1,\ldots,n\}$ (and $\pi^{\zeta,\eta_i}$ optimal with marginals $\zeta$ and $\eta_i$). Developing the right-hand side expression in the above inequality, we get
\begin{align*}
\gamma\langle \nabla E(\mu)-&\nabla E(\zeta),\mu-\zeta\rangle\\
&\le -\frac{1}{n}\sum_{i=1}^n\left[\iint(\phi^{\mu,\nu_i})^{cc}(x)d\pi^{\mu,\nu_i}(x,y)+\iint(\phi^{\zeta,\eta_i})^{cc}(x)d\pi^{\zeta,\eta_i}(x,y)\right]\\&\,\,\,+\frac{1}{n}\sum_{i=1}^n\left[\iint(\phi^{\mu,\nu_i})^{cc}(x)d\pi^{\zeta,\eta_i}(x,y)+\iint(\phi^{\zeta,\eta_i})^{cc}(x)d\pi^{\mu,\nu_i}(x,y)\right].
\end{align*}

From the  condition \eqref{eq:condW} in the Kantorovich's dual problem, we have that \linebreak $(\phi^{\mu,\nu_i})^{cc}(x)\le\vert x-y\vert^2 - (\phi^{\mu,\nu_i})^{c}(y)$ and $(\phi^{\zeta,\eta_i})^{cc}(x)\le\vert x-y\vert^2 - (\phi^{\zeta,\eta_i})^{c}(y)$ for all $i\in\{1,\ldots,n\}$. 
Moreover, we have that  $(\phi^{\mu,\nu_i})^{cc}(x)d\pi^{\mu,\nu_i}(x,y)= \left( \vert x-y\vert^2- (\phi^{\mu,\nu_i})^{c}(y)\right)\linebreak d\pi^{\mu,\nu_i}(x,y)$ and likewise $(\phi^{\zeta,\eta_i})^{cc}(x)d\pi^{\zeta,\eta_i}(x,y)= \left(\vert x-y\vert^2- (\phi^{\zeta,\eta_i})^{c}(y)\right)d\pi^{\zeta,\eta_i}(x,y)$. 
We therefore deduce that
\begin{align*}
\gamma\langle \nabla E(\mu)-\nabla E(\zeta),\mu-\zeta\rangle & \le -\frac{1}{n}\sum_{i=1}^n\left[\iint\left(\vert x-y\vert^2-(\phi^{\mu,\nu_i})^{c}(y)\right)d\pi^{\mu,\nu_i}(x,y)\right.\\ 
&+\left. \iint\left(\vert x-y\vert^2-(\phi^{\zeta,\eta_i})^{c}(y)\right)d\pi^{\zeta,\eta_i}(x,y)\right]\\
& +\frac{1}{n}\sum_{i=1}^n\left[\iint\left(\vert x-y\vert^2-(\phi^{\mu,\nu_i})^{c}(y)\right)d\pi^{\zeta,\eta_i}(x,y)\right.\\
&+\left.\iint\left(\vert x-y\vert^2-(\phi^{\zeta,\eta_i})^{c}(y)\right)d\pi^{\mu,\nu_i}(x,y)\right]\\
 & = \frac{1}{n}\sum_{i=1}^n\int_{\Omega}\left[(\phi^{\mu,\nu_i})^{c}(y)-(\phi^{\zeta,\eta_i})^{c}(y)\right]d(\nu_i-\eta_i)(y).
% \le &-\frac{1}{n}\sum_{i=1}^n\iint\left[-(\phi^{\mu,\nu_i})^{c}(y)+\vert x-y\vert^2\right]d\pi^{\mu,\nu_i}(x,y)-\frac{1}{n}\sum_{i=1}^n\iint\left[-(\phi^{\zeta,\eta_i})^{c}(y)+\vert x-y\vert^2\right]d\pi^{\zeta,\eta_i}(x,y)\\
%&+\frac{1}{n}\sum_{i=1}^n\iint\left[-(\phi^{\mu,\nu_i})^{c}(y)+\vert x-y\vert^2\right]d\pi^{\zeta,\eta_i}(x,y)\\&+\frac{1}{n}\sum_{i=1}^n\iint\left[-(\phi^{\zeta,\eta_i})^{c}(y)+\vert x-y\vert^2\right]d\pi^{\mu,\nu_i}(x,y)\\
% = &\frac{1}{n}\sum_{i=1}^n\int_{\Omega}\left[(\phi^{\mu,\nu_i})^{c}(y)-(\phi^{\zeta,\eta_i})^{c}(y)\right]d(\nu_i-\eta_i)(y).
\end{align*}
For all $1\le i\le n$, we have that $(\phi^{\mu,\nu_i})^{c}$ and $(\phi^{\zeta,\eta_i})^{c}$ are $2\diam(\Omega)$-Lipschitz by definition, which implies that $\left[(\phi^{\mu,\nu_i})^{c}-(\phi^{\zeta,\eta_i})^{c}\right]$ is $4\diam(\Omega)$-Lipschitz for all $1\le i\le n$. 
We then conclude
\begin{align*}
\gamma\langle \nabla E(\mu)-&\nabla E(\zeta),\mu-\zeta\rangle \\
& \le \frac{4\diam(\Omega)}{n}\sum_{i=1}^n\sup\left\{\int\phi\ d(\nu_i-\eta_i);\ \phi\in\cap\mathbb{L}^1(\vert\nu_i-\eta_i\vert),\| \phi\|_{Lip}\le 1\right\}\\
& =\frac{4\diam(\Omega)}{n}\sum_{i=1}^n W_1(\nu_i,\eta_i)\le\frac{4\diam(\Omega)}{n}\sum_{i=1}^n W_2(\nu_i,\eta_i),
\end{align*}
by the Kantorovich-Rubinstein theorem presented in \cite{villani2003topics}, while the last inequality above comes from H\"older inequality between the distance $W_2$ and the distance $W_1$ defined  for $\theta_1,\theta_2$  (probability measures on $\Omega$ with moment of order $1$) as
\[W_1(\theta_1,\theta_2)=\underset{\pi}{\inf}\int_{\Omega}\int_{\Omega}\vert x-y\vert d\pi(x,y)\]
where $\pi$ is a probability measure on $\Omega\times\Omega$ with marginals $\theta_1$ and $\theta_2$. 
Since $\mu$ and $\zeta$ are independent, we can assign to $\nu_i$ any $\eta_{\sigma(i)}$ for $\sigma\in\mathcal{S}_n$ the permutation group of $\{1,\ldots,n\}$ to  obtain
$
\gamma\langle \nabla E(\mu)-\nabla E(\zeta),\mu-\zeta\rangle\le\frac{4\diam(\Omega)}{n}\underset{\sigma\in\mathcal{S}_n}{\inf}\sum_{i=1}^n W_2(\nu_i,\eta_{\sigma(i)}),
$
which completes the proof.
\end{proof}

\subsection{Proof of convergence properties}\label{App_conv}
\paragraph{Convergence of $\mu_{\mathbb{P}}^{\gamma}$ towards $\mu_{\mathbb{P}}^{0}$}
\begin{proof}[Proof of Theorem \ref{th:biais}]
By Theorem 2.1.(d) in \cite{braides2006handbook}, $J_{\P}^{\gamma} \ \Gamma$-converges to $J_{\P}^0$ in $2$-Wasserstein metric. Indeed for every sequence $(\mu_{\gamma})_{\gamma}\subset\PP_{2}(\Omega)$ converging to $\mu\in\PP_{2}(\Omega)$,
$$J_{\P}^0(\mu)\le\underset{\gamma\rightarrow 0}{\liminf} \ J_{\P}^{\gamma}(\mu_{\gamma})$$
by lower semicontinuity of $J_{\P}^{\gamma}$ with respect to the $W_2$ metric. Moreover, there exists a sequence $(\mu_{\gamma})_{\gamma}$ converging to $\mu$ (for instance take $(\mu_{\gamma})_{\gamma}$ constant and equal to $\mu$) such that
$\underset{\gamma\rightarrow 0}{\lim}J_{\P}^{\gamma}(\mu_{\gamma})=\underset{\gamma\rightarrow 0}{\lim}J_{\P}^{\gamma}(\mu)=J_{\P}^0(\mu).$
One can also notice that $J_{\P}^{\gamma}:\PP_{2}(\Omega)\rightarrow\R$ is equi-coercive: for all $t\in\R$, the set $\{\nu\in\PP_{2}(\Omega)\ \text{such that}\ J_{\P}^{\gamma}(\nu)\le t\}$ is included in a compact $K_t$ since it is closed in the compact set $\PP_{2}(\Omega)$ (by compactness of $\Omega$). Therefore, we can apply the fundamental theorem of $\Gamma$-convergence (Theorem 2.10 in \cite{braides2006handbook})  in the metric space $(\PP_{2}(\Omega),W_2)$ to obtain the first statement of Theorem \ref{th:biais}.

Let us now use this result to prove the convergence in non-symmetric Bregman divergence of $\mu_{\P}^{\gamma}$ under the assumption that the population Wasserstein barycenter is unique.  By definition \eqref{fP_gamma} of $\mu_{\P}^{\gamma}$, we get that
\begin{equation}
\label{biais1}
\int W_2^2(\mu_{\P}^{\gamma},\nu)d\P(\nu)-\int W_2^2(\mu_{\P}^0,\nu)d\P(\nu)+\gamma(E(\mu_{\P}^{\gamma})-E(\mu_{\P}^0))\le 0,
\end{equation}
and by definition \eqref{fP_0} of $\mu_{\P}^0$, one has that $\int W_2^2(\mu_{\P}^{\gamma},\nu)d\P(\nu)-\int W_2^2(\mu_{\P}^0,\nu)d\P(\nu)\ge0$.
Therefore, one has that $E(\mu_{\P}^{\gamma})-E(\mu_{\P}^0) \leq 0$ and thus, by definition \eqref{eq:nonsymDE} of the non-symmetric Bregman divergence, it follows that
\begin{align*}
D_E(\mu_{\P}^{\gamma},\mu_{\P}^0)&\le\langle\nabla E(\mu_{\P}^0),\mu_{\P}^0-\mu_{\P}^{\gamma}\rangle.
\end{align*}
Since $\nabla E(\mu_{\P}^0)$ is assumed to be a continuous function on the compact set $\Omega$, the above inequality and the fact that $\lim\limits_{\gamma \rightarrow 0} W_2(\mu_{\P}^{\gamma},\mu_{\P}^0)=0$ implies that $\lim\limits_{\gamma \rightarrow 0} D_E(\mu_{\P}^{\gamma},\mu_{\P}^0)= 0$ since convergence of probability measures for the $W_2$ metric implies weak convergence.
\end{proof} 

\paragraph{Convergence of $\bmu_{\mathbb{P}_n}^{\gamma}$ towards $\mu_{\mathbb{P}}^{\gamma}$.}
In what follows, $C$ denotes a universal constant whose value may change from line to line.
\begin{proof}[Proof of Theorem \ref{th:convergence_var}]
We denote by $C$ a universal constant whose value may change from line to line. From the subgradient's inequality \eqref{ineg_subgradient} and following the same process used in the proof of the stability's Theorem \ref{stability}, we have that, for each  $\bnu_i,\ i=1,\ldots,n$, there exists $\phi^{\bmu_{\P_n}^{\gamma},\bnu_i}$ integrable with respect to $\bmu_{\P_n}^{\gamma}(x)dx$ such that for all $\eta\in \PP_{2}(\Omega)$:
\begin{equation}
\label{ineg_cont}
\left\langle\frac{1}{n}\sum_{i=1}^n\phi^{\bmu_{\P_n}^{\gamma},\bnu_i}+\gamma\nabla E(\bmu_{\P_n}^{\gamma}),\eta-\bmu_{\P_n}^{\gamma}\right\rangle\ge 0.
\end{equation}
By applying once again the subgradient's inequality, we get
\begin{align*}
\mu_{\P}^{\gamma} \ \text{minimizes} \ J_{\P}^{\gamma} \ \Leftrightarrow \ & \exists  \phi\in\partial J_{\P}^{\gamma}(\mu_{\P}^{\gamma}) \ \text{s. t.}  \ \langle\phi,\eta-\mu_{\P}^{\gamma}\rangle\ge 0 \ \text{for all} \ \eta\in\PP_{2}(\Omega).
\end{align*}
Let us explicit the form of a subgradient $\phi\in\partial J_{\P}^{\gamma}(\mu_{\P}^{\gamma})$ using again the Theorem of the subdifferential of a sum. We have that $\mu\mapsto W_2^2(\mu,\nu)$ is continuous for all $\nu\in \PP_2(\Omega)$. Moreover by symmetry, $\nu\mapsto W_2^2(\mu,\nu)$ is measurable for all $\mu\in\PP_{2}(\Omega)$ and
$
W_2^2(\mu,\nu) \le\iint\vert x-y\vert^2d\mu(x)d\nu(y)\le 2\int\vert x\vert^2 d\mu(x) + 2\int\vert y\vert^2d\nu(y)\le C 
$
is integrable with respect to $d\P(\nu)$. Hence, by the Theorem of continuity under integral sign, we deduce that $\mu\mapsto\E[W_2^2(\mu,\bnu)]$ is continuous. Thus we can manage the subdifferential of the following sum and one has that $\partial J_{\P}^{\gamma}(\mu_{\P}^{\gamma})=\partial_1[\E(W_2^2(\mu_{\P}^{\gamma},\bnu))]+\gamma \nabla E(\mu_{\P}^{\gamma})$, where $\bnu$ is still a random measure with distribution $\P$. Also the Theorem 23 in \cite{rockafellar1974conjugate} implies
$
\partial_1\E[W_2^2(\mu_{\P}^{\gamma},\bnu)]=\E[\partial_1W_2^2(\mu_{\P}^{\gamma},\bnu)].
$
We sum up as
\begin{equation} \label{ineg_cont2}
\mu_{\P}^{\gamma} \ \text{minimizes} \ J_{\P}^{\gamma}  \Leftrightarrow  \left\langle\int\phi^{\mu_{\P}^{\gamma},\nu}d\P(\nu)+\gamma\nabla E(\mu_{\P}^{\gamma}),\eta-\mu_{\P}^{\gamma}\right\rangle\ge 0,\, \forall \ \eta\in\PP_{2}(\Omega). 
\end{equation}

In the sequel, to simplify the notation, we use $\bmu:=\bmu_{\P_n}^{\gamma}$ and $\eta:=\mu_{\P}^{\gamma}$. Therefore thanks to \eqref{ineg_cont} and  \eqref{ineg_cont2}
\begin{align}
d_E(\bmu,\eta) &=\langle\nabla E(\bmu)-\nabla E(\eta),\bmu-\eta\rangle\nonumber\\
 & \le - \frac{1}{\gamma}\left\langle\frac{1}{n}\sum_{i=1}^n\phi^{\bmu,\bnu_i}-\int\phi^{\eta,\nu}d\P(\nu),\bmu-\eta\right\rangle\label{eq1}\\
\nonumber
& =\frac{1}{\gamma}\left( \frac{1}{n}\sum_{i=1}^n\left[\int\phi^{\bmu,\bnu_i}(x)d\eta(x) -\int\phi^{\bmu,\bnu_i}(x)d\bmu(x) \right]\right.\\
& \hspace{2cm}\left.+\iint\phi^{\eta,\nu}d\P(\nu)d\bmu(x)-\iint\phi^{\eta,\nu}d\P(\nu)d\eta(x) \right).\nonumber
\end{align}
We would like to switch integrals of the two last terms. In that purpose, we use that $\int W_2^2(\eta,\nu) d\P(\nu) <+\infty$, since $\P\in W_2(\P_2(\Omega))$. 

As  $0\le \int W_2^2(\eta,\nu)  d\P(\nu) =\int \left(  \int\phi^{\eta,\nu}(x)d\eta(x)+\\ \int\psi^{\eta,\nu}(x)d\nu(y) \right) d\P(\nu) $, we also have that $\iint\phi^{\eta,\nu}(x)d\eta(x)d\P(\nu)<+\infty$. Since $x\mapsto\phi^{\eta,\nu}(x)$ and $\nu\mapsto\phi^{\eta,\nu}(x)$ are measurables, we obtain by Fubini's theorem
$
\int_{\Omega}\int_{\PP_2(\Omega)}\phi^{\eta,\nu}d\P(\nu)d\eta(x)=\linebreak\int_{\PP_2(\Omega)}\int_{\Omega}\phi^{\eta,\nu}d\eta(x)d\P(\nu).
$
By the same tools, since
\begin{align*}
\int  W_2^2(\bmu,\nu) d\P(\nu) &=  \int \left( \int\phi^{\bmu,\nu}(x)d\bmu(x)+\int\psi^{\bmu,\nu}(x)d\nu(y)\right) d\P(\nu)\\
& \ge \int  \left(\int\phi^{\eta,\nu}(x)d\bmu(x)+\int\psi^{\eta,\nu}(x)d\nu(y)\right) d\P(\nu),
\end{align*}
we get  $\int \left( \int\phi^{\eta,\nu}(x)d\mu(x) \right) d\P(\nu) <+\infty$, so
$
\int_{\Omega}\int_{\PP_2(\Omega)}\phi^{\eta,\nu}d\P(\nu)d\bmu(x)=\linebreak \int_{\PP_2(\Omega)}\int_{\Omega}\phi^{\eta,\nu}d\bmu(x)d\P(\nu).
$

\noindent
Therefore, by the dual formulation of Kantorovich, we have that
\begin{align}
\label{eq2}
-\int\phi^{\bmu,\bnu_i}d\bmu(x) &= \int\psi^{\bmu,\bnu_i}(y)d\bnu_i(y)-\iint\vert x-y\vert^2d\pi^{\bmu,\bnu_i}(x,y)\\
\label{eq3}
-\int\phi^{\eta,\nu}d\eta(x) &= \int\psi^{\eta,\nu}(y)d\nu(y)-\iint\vert x-y\vert^2d\pi^{\eta,\nu}(x,y)
\end{align}
where $\pi^{\bmu,\bnu_i}$ and $\pi^{\eta,\nu}$ are optimal transport plans for the Wasserstein distance. Also, $\phi^{\bmu,\bnu_i}$ and $\phi^{\eta,\nu}$ verify the Kantorovich condition, that is
\begin{align}
\label{eq4}
\phi^{\bmu,\bnu_i}(x)&\le-\psi^{\bmu,\bnu_i}(y)+\vert x-y\vert^2\\
\label{eq5}
\phi^{\eta,\nu}(x)&\le-\psi^{\eta,\nu}(y)+\vert x-y\vert^2.
\end{align}

Next, the trick is to write $\int\phi^{\bmu,\bnu_i}(x)d\eta(x)=\iint\phi^{\bmu,\bnu_i}(x)d\pi^{\eta,\bnu_i}(x,y)$ and \linebreak $\int\phi^{\eta,\nu}(x)d\bmu(x)=\iint\phi^{\eta,\nu}(x)d\pi^{\bmu,\nu}(x,y)$. Thus, by using the equalities \eqref{eq2}, \eqref{eq3} and the inequalities \eqref{eq4} and \eqref{eq5}, the result \eqref{eq1} becomes
\begin{equation}
\begin{split}
\label{ref_ineq}
\gamma d_E(\bmu,\eta) \le  &-\frac{1}{n}\sum_{i=1}^n\iint\vert x-y\vert^2d\pi^{\bmu,\bnu_i}(x,y)+\frac{1}{n}\sum_{i=1}^n\iint\vert x-y\vert^2d\pi^{\eta,\bnu_i}(x,y) \\
&  +\int\iint\vert x-y\vert^2d\pi^{\bmu,\nu}(x,y)d\P(\nu) - \int\iint\vert x-y\vert^2d\pi^{\eta,\nu}(x,y)d\P(\nu).
% = & \E\left[\iint\vert x-y\vert^2d\pi^{\bmu,\bnu}(x,y)\right]\ -\frac{1}{n}\sum_{i=1}^n\iint\vert x-y\vert^2d\pi^{\bmu,\bnu_i}(x,y)\\
% & \ +\frac{1}{n}\sum_{i=1}^n\iint\vert x-y\vert^2d\pi^{\eta,\bnu_i}(x,y) - \E\left[\iint\vert x-y\vert^2d\pi^{\eta,\bnu}(x,y)\right].
\end{split}
\end{equation}
We denote
\begin{align}
S_{\bmu_{\P_n}^{\gamma}}^n &:=\int \iint\vert x-y\vert^2d\pi^{\bmu_{\P_n}^{\gamma},\nu}(x,y) d\P(\nu)  \ -\frac{1}{n}\sum_{i=1}^n\iint\vert x-y\vert^2d\pi^{\bmu_{\P_n}^{\gamma},\bnu_i}(x,y)  \label{def:Snn} \\
S_{\mu_{\P}^{\gamma}}^n &:=\frac{1}{n}\sum_{i=1}^n\iint\vert x-y\vert^2d\pi^{\mu_{\P}^{\gamma},\bnu_i}(x,y) - \E\left(\iint\vert x-y\vert^2d\pi^{\mu_{\P}^{\gamma},\bnu}(x,y)\right)\label{def:Sn},
\end{align}
and the previous inequality \eqref{ref_ineq} finally  writes
\begin{equation}
\label{ineg1}
\gamma d_E(\bmu_{\P_n}^{\gamma},\mu_{\P}^{\gamma})\le S_{\bmu_{\P_n}^{\gamma}}^n+S_{\mu_{\P}^{\gamma}}^n.
\end{equation}

\noindent
Taking the expectation with respect to the random measures, \eqref{ineg1} implies
\begin{equation}
\label{variance}
\gamma^2\E(d_E^2(\bmu_{\P_n}^{\gamma},\mu_{\P}^{\gamma}))\le 2\E(\vert S_{\bmu_{\P_n}^{\gamma}}^n\vert^2)+2\E(\vert S_{\mu_{\P}^{\gamma}}^n\vert^2).
\end{equation}

The first term related to ${\mu_{\P}^{\gamma}}^n$ is easy to handle, since for $i=1,\ldots,n$ the random variables $\iint\vert x-y\vert^2d\pi^{\mu_{\P}^{\gamma},\bnu_i}(x,y)$ are independent and identically distributed. From the law of large numbers, we can notice that $S_{\mu_{\P}^{\gamma}}^n\longrightarrow 0$ almost surely when $n\rightarrow +\infty$. In particular, we observe that
\begin{equation}
\label{var_carre}
\E\left(\vert S_{\mu_{\P}^{\gamma}}^n\vert^2\right)=\frac{1}{n}\text{Var}\left(\iint\vert x-y\vert^2d\pi^{\mu_\P^{\gamma},\bnu}(x,y)\right)\leq\frac{C}{n}.
\end{equation}
\noindent
Let us now study $\E(\vert S_{\bmu_{\P_n}^{\gamma}}^n\vert^2)$ thanks to the empirical process theory. We recall that  the class of functions $\mathcal{H}$ on $\PP_2(\Omega)$ is defined as 
\begin{equation} \label{eq:classH}
\mathcal{H}=\{h_{\mu}:\nu\in\PP_2(\Omega)\mapsto W_2^2(\mu,\nu)\in\R;\mu\in\PP_2(\Omega)\}
\end{equation}
ant its associated norm is  $\Vert G\Vert_{\mathcal{H}}:={\sup}_{h\in\mathcal{H}}\vert G(h)\vert$ where $G:\mathcal{H}\rightarrow\R$. 

\noindent
Therefore we obtain
\begin{align}
S_{\bmu_{\P_n}^{\gamma}}^n % &=\int  W_2^2(\bmu_{\P_n}^{\gamma},\nu) d\P(\nu) -\frac{1}{n}\sum_{i=1}^n W_2^2(\bmu_{\P_n}^{\gamma},\bnu_i) \nonumber \\
&= \int_{\PP_2(\Omega)} h_{\bmu_{\P_n}^{\gamma}}(\nu)d\P(\nu)-\int_{\PP_2(\Omega)} h_{\bmu_{\P_n}^{\gamma}}(\nu)d\P_n(\nu):=(\P-\P_n)(h_{\bmu_{\P_n}^{\gamma}})\\
& \le \underset{h\in\mathcal{H}}{\sup}\vert(\P-\P_n)\left(h\right)\vert. \label{S_n_gamma}
\end{align}
We define the envelope function of $\mathcal{H}$ by 
$$H:\nu\in\PP_2(\Omega)\mapsto{\sup}_{\mu\in\PP_2(\Omega)}\{W_2(\mu,\nu);W_2^2(\mu,\nu)\},$$ which is integrable with respect to $\P$ by compacity of $\Omega$.
Let then $\mathcal{H}_M$ be the class of functions $\tilde{h}_{\mu}:= h_{\mu}\mathds{1}_{H\leq M}$ when $h_{\mu}$ ranges over $\mathcal{H}$. By the triangle reverse inequality, we have for $\tilde{h}_{\mu},\tilde{h}_{\mu'}\in\mathcal{H}_M$
\begin{align*}
\Vert \tilde{h}_{\mu}-\tilde{h}_{\mu'} \Vert_{\mathbb{L}_1(\P_n)}&=\frac{1}{n}\sum_{i=1}^n\vert W_2(\mu,\nu_i)- W_2(\mu',\nu_i)\vert \ (W_2(\mu,\nu_i)- W_2(\mu',\nu_i))\mathds{1}_{H\leq M}\\
&\leq W_2(\mu,\mu')\frac{2}{n}\sum_{i=1}^n H(\nu_i)\mathds{1}_{H\leq M} \leq 2M W_2(\mu,\mu').
\end{align*}
We deduce that $N(\varepsilon,\mathcal{H}_M,\mathbb{L}_1(\P_n))\leq N(\frac{\varepsilon}{2M},K_M,W_2)$ where $K_M=\{\mu\in\PP_2(\Omega)\}$ is compact. Then from Borel-Lebesgue, we deduce that $\log N(\varepsilon,\mathcal{H}_M,\mathbb{L}_1(\P_n))$ can be bounded from above by a finite number which does not depend on $n$. Theorem 2.4.3 in \cite{van1996weak} allows us to conclude that $\vert S_{\bmu_{\P_n}^{\gamma}}^n\vert$ tends to $0$ almost surely. By the mapping theorem, $\vert S_{\bmu_{\P_n}^{\gamma}}^n\vert^2$ also tends to $0$ a.s. Since it is bounded by a constant dependant only on the diameter of $\Omega$, we have that it is bounded by an integrable function. By the theorem of dominated convergence, we get $\E\left(\vert S_{\bmu_{\P_n}^{\gamma}}^n\vert^2\right)\underset{n\rightarrow \infty}{\longrightarrow} 0.$ Gathering \eqref{variance}, \eqref{var_carre}, we get for all $\gamma>0$
$$\E(d_E^2(\bmu_{\P_n}^{\gamma},\mu_{\P}^{\gamma}))\ \underset{n\rightarrow\infty}{\longrightarrow} 0.$$
\end{proof}
\noindent
{\bf Rate of convergence between $\bmu_{\mathbb{P}_n}^{\gamma}$ and $\mu_{\mathbb{P}}^{\gamma}$}
In order to achieve a rate a convergence, we will need existing results on the notion of bracketing number defined below.
\begin{defi}
Given two real-valued functions $l$ and $r$, the bracket $[l,r]$ is the set of all functions $f$ with 
$l \le f\le r$. An $\epsilon$-bracket is a bracket $[l,r]$ with $\Vert l-r \Vert < \epsilon$. The bracketing number $N_{[]}(\epsilon,\mathcal{F},\Vert \cdot\Vert)$ is the minimum number of $\epsilon$-brackets needed to cover $\mathcal{F}$.
\end{defi}

\begin{proof}[Proof of Theorem \ref{rate_convergence}]
This proof follows on from the proof of Theorem \ref{th:convergence_var}. We recall from \eqref{variance}
\begin{equation}
\gamma^2\E(d_E^2(\bmu_{\P_n}^{\gamma},\mu_{\P}^{\gamma}))\le 2\E(\vert S_{\bmu_{\P_n}^{\gamma}}^n\vert^2)+2\E(\vert S_{\mu_{\P}^{\gamma}}^n\vert^2), \ \mbox{for} \ S_{\bmu_{\P_n}^{\gamma}}^n,S_{\mu_{\P}^{\gamma}}^n \ \mbox{defined in} \ \eqref{def:Snn},\eqref{def:Sn}
\end{equation}
where by \eqref{var_carre}, $\E\left(\vert S_{\mu_{\P}^{\gamma}}^n\vert^2\right)\leq\frac{C}{n}$ and by \eqref{S_n_gamma} we have for $\mathcal{H}$ given in \eqref{eq:classH} 
$$\vert S_{\bmu_{\P_n}^{\gamma}}^n \vert \leq \underset{h\in\mathcal{H}}{\sup}\vert(\P-\P_n)\left(h\right)\vert.$$
Rewritting this term, we get $\vert S_{\bmu_{\P_n}^{\gamma}}^n \vert \leq \frac{1}{\sqrt{n}}\Vert\G_n\Vert_{\mathcal{H}}$ where $\G_n(h) =\sqrt{n}(\P_n-\P)(h) $. We  obtain
\begin{equation}
\label{ineg2}
\E\left(\vert S_{\bmu_{\P_n}^{\gamma}}^n\vert^2\right)\le\frac{1}{n}\E\left(\Vert\G_n\Vert_{\mathcal{H}}^2\right)
=\frac{1}{n}\Vert\ \Vert\G_n\Vert_{\mathcal{H}}\Vert_{\mathbb{L}_2
(\P)}^2.
\end{equation}
We then use  the following Theorem 2.14.1. of \cite{van1996weak}  to control the last term in \eqref{ineg2}.
\begin{thm}
Let $\mathcal{H}$ be a $Q$-measurable class of measurable functions with measurable envelope function $H$. Then for $p\ge 1$,
\begin{equation}
\label{entropy}
\Vert\ \Vert\G_n\Vert_{\mathcal{H}}\Vert_{\mathbb{L}_p
(Q)}\le CI(1,\mathcal{H})\Vert H\Vert_{\mathbb{L}_{2\vee p}(Q)}
\end{equation}
with $C$ a constant, $I(1,\mathcal{H})$ defined in \eqref{def_metric_entropy} and $H$ an envelope function.
\end{thm}
Gathering the results of \eqref{variance}, \eqref{var_carre}, \eqref{ineg2} and \eqref{entropy}, we get
\begin{equation}
\E(d_E^2\left(\bmu_{\P_n}^{\gamma},\mu_{\P}^{\gamma})\right)\le \frac{1}{\gamma^2n}\left(C+CI(1,\mathcal{H})\Vert H\Vert_{\mathbb{L}_2(\P)}\right)
\end{equation}
which is completly valid for any $\Omega$ compact in $\R^d$. The norm $\Vert H\Vert_{\mathbb{L}_2(\P)}$ is clearly finite since for all $\nu\in \PP_2(\Omega)$, $\vert h_{\mu}(\nu)\vert\leq 4c_{\Omega}^2$, with $c_{\Omega}^2=\underset{x\in \Omega}{\sup}\vert x\vert^2$. 
\end{proof}

\begin{proof}[Proof of the Theorem \ref{th:rate1D}]
We assume here that $\Omega\subset\R$ compact. It remains to study the term $I(1,\mathcal{H})$ defined in \eqref{def_metric_entropy} for $\mathcal{H}$  in \eqref{eq:classH}.
By the triangle reverse inequality, we have 
\begin{align*}
\vert h_{\mu}(\nu)-h_{\mu'}(\nu)\vert=& \ \vert W_2(\nu,\mu)-W_2(\nu,\mu')\vert \ (W_2(\nu,\mu)+W_2(\nu,\mu')) 
 \le \  W_2(\mu,\mu') \ 2 H(\nu).
\end{align*}
Then, from Theorem 2.7.11 in \cite{van1996weak}, and since Theorem 4 in \cite{kolmogorov1959varepsilon} allows us to bound the metric entropy by the bracket entropy, we get
\begin{align}
\log N(\epsilon \Vert H\Vert_{\mathbb{L}_2(Q)},\mathcal{H},\Vert \cdot\Vert_{\mathbb{L}_2(Q)})\le &
\log N_{[]}(\epsilon \Vert H\Vert_{\mathbb{L}_2(Q)},\mathcal{H},\Vert \cdot\Vert_{\mathbb{L}_2(Q)})\nonumber \\
\le & \log N(\epsilon,\PP_2(\Omega),W_2) \le \log N_{[]}(\epsilon,\PP_2(\Omega),W_2) \label{metric_entrop}.
\end{align}
Also,  for $d=1$, we have
\begin{equation}
\label{ineg}
W_2(\mu,\mu')  =\left(\int_0^1\vert F_{\mu}^-(t)-F_{\mu'}^-(t)\vert^2dt\right)^{1/2}= \Vert F_{\mu}^--F_{\mu'}^-\Vert_{\mathbb{L}_2([0,1])}
\end{equation}
where $F_{\mu}^-$ is the quantile function of the cumulative distribution function $F_{\mu}$ of $\mu$. We denote by $\mathcal{G}=\{F_{\mu}^-,\mu\in \PP_2(\Omega)\}$ the class of quantile functions of probability measures $\mu$ in $\PP_2(\Omega)$, which are monotonic functions. Moreover, we can observe that $F_{\mu}^-:[0,1]\rightarrow[F_{\mu}^-(0),F_{\mu}^-(1)]\subseteq \Omega$, where $\Omega$ is a compact included in $\R$. Hence, $\mathcal{G}$ is uniformly bounded, say by a constant $M > 0$. By Theorem 2.7.5. of \cite{van1996weak} concerning the bracket entropy of the class of monotonic functions, we obtain that
$
\log N_{[]}(\epsilon,\mathcal{G},\mathbb{L}_2[0,1])\le \frac{CM}{\epsilon},
$
for some constant $C > 0$. Finally, from relations \eqref{metric_entrop} and \eqref{ineg}, we can deduce that
$$I(1,\mathcal{H})=\underset{Q}{\sup}\int_0^1\sqrt{1+\log N(\epsilon\Vert H\Vert_{\mathbb{L}_2(Q)},\mathcal{H},\mathbb{L}_2(Q))}d\epsilon \le\int_0^1\sqrt{1+\frac{CM}{\epsilon}}d\epsilon<\infty.$$
\end{proof} 

\begin{proof}[Proof of the Theorem \ref{th:rateSobolev}]
We here consider that  $\Omega$ is a compact of $\R^d$ and $E$ is given by \eqref{Sobolev_penalty}. Let us begin by underlining that since the norm of a Sobolev space is weakly$^*$ lower semicontinuous, $E$ is indeed lower semicontinuous for the Wasserstein metric. Supposing that $\Omega$ has a $C^1$ boundary, we have by the Sobolev embedding theorem that $H^k(\Omega)$ is included in the H\"older space $C^{m,\beta}(\bar{\Omega})$ for any integer $m$ and $\beta \in ]0,1]$ satisfying $m + \beta = k-d/2$. Hence, the densities of $\bmu_{\P_n}^{\gamma}$ and $\mu_{\P}^{\gamma}$ given by \eqref{fPn_gamma} and \eqref{fP_gamma} belong to $C^{m,\beta}(\bar{\Omega})$.

\noindent
From the Theorem \ref{rate_convergence}, we will use that:
 $$\E(d_E^2\left(\bmu_{\P_n}^{\gamma},\mu_{\P}^{\gamma})\right)\le \frac{1}{\gamma^2n}\left(C+CI(1,\mathcal{H})\Vert H\Vert_{\mathbb{L}_2(\P)}\right).$$
Arguing similarly, we have $\Vert H\Vert_{\mathbb{L}_2(Q)}<\infty$, where $H(\nu) = \underset{\mu\in\DD(E)}{\sup}\{W_2(\mu,\nu);W_2^2(\mu,\nu)\}$ where $\DD(E)$ is defined by \eqref{eq:domE}. Thus, instead of controlling the metric entropy $N(\epsilon \Vert H\Vert_{\mathbb{L}_2(Q)},\mathcal{H},\Vert \cdot\Vert_{\mathbb{L}_2(Q)})$, it is enough to bound the metric entropy  $N(\epsilon,\DD(E),W_2)$ thanks to Theorem 2.7.11 in \cite{van1996weak}.

To this end, since $\mu,\mu'\in\DD(E)$ are a.c.\ measures, one has that 
$$
W_2(\mu,\mu')\le \left(\int_{\Omega}\vert T(x)-T'(x)\vert^2 dx\right)^{1/2} \mbox{where}\ T\#\lambda^d=\mu \ \mbox{and} \ T'\#\lambda^d=\mu',
$$
with $\lambda^d$ denoting the Lebesgue measure on $\Omega$. Thanks to Theorem 3.3 in \cite{de2014monge} on the regularity of optimal maps (results initally due to Caffarelli, \cite{caffarelli1992} and \cite{caffarelli1996}), the coordinates of $T$ and $T'$ are $C^{m+1,\beta}(\bar{\Omega})$ functions $\lambda^d-a.e$. Thus, we can bound $N(\epsilon,\DD(E),W_2)$ by the bracket entropy  $N_{[]}(\epsilon,C^{m+1,\beta}(\bar{\Omega}),\mathbb{L}_2(\Omega))$ since $\vert T(x)-T'(x)\vert^2\linebreak =\sum_{j=1}^d\vert T_j(x_j)-T'_j(x_j)\vert^2$ where $T_j,T'_j:\Omega\rightarrow\R$. Now, by Corollary 2.7.4 in \cite{van1996weak},  
$$ \log N_{[]}(\epsilon,C^{m+1,\beta}(\bar{\Omega}),\mathbb{L}_2(\Omega))\le K\left(\frac{1}{\epsilon}\right)^V$$ for any $V\ge d/(m+1)$. Hence, as soon as $V/2<1$ (for which the condition $k>d-1$ is sufficient if $V =  d/(m+1)$), the upper bound in \eqref{rate_conv1} is finite for $\mathcal{H}=\{h_{\mu}:\nu\in\PP_2(\Omega)\mapsto W_2^2(\mu,\nu)\in\R;\mu\in\DD(E)\} $, which yields the result of Theorem \ref{th:rateSobolev} by finally following the arguments in the proof of Theorem \ref{th:rate1D} and since $d_{E_G}\leq d_E$. 
\end{proof}

\bibliographystyle{alpha}

\begin{thebibliography}{{\'A}EdBCAM15}

\bibitem[AB06]{aliprantis2006infinite}
C.~D. Aliprantis and K.~Border.
\newblock {\em Infinite dimensional analysis: a {H}itchhiker's guide}.
\newblock Springer Science \& Business Media, 2006.

\bibitem[AC11]{agueh2011barycenters}
M.~Agueh and G.~Carlier.
\newblock Barycenters in the {W}asserstein space.
\newblock {\em SIAM Journal on Mathematical Analysis}, 43(2):904--924, 2011.

\bibitem[AC17]{AGUEH2017812}
M.~Agueh and G.~Carlier.
\newblock Vers un th{\'e}or{\`e}me de la limite centrale dans l'espace de
  {W}asserstein?
\newblock {\em Comptes Rendus Math\'ematique}, 355(7):812--818, 2017.

\bibitem[{\'A}EdBCAM15]{alvarez2015wide}
P.~{\'A}lvarez-Esteban, E.~del Barrio, J.~Cuesta-Albertos, and C.~Matr{\'a}n.
\newblock Wide consensus for parallelized inference.
\newblock {\em arXiv e-prints}, 1511.05350, 2015.

\bibitem[AG18]{ambrosio2018finer}
Luigi Ambrosio and Federico Glaudo.
\newblock Finer estimates on the 2-dimensional matching problem.
\newblock {\em arXiv preprint arXiv:1810.07002}, 2018.

\bibitem[AGS08]{ambrosio2008gradient}
L.~Ambrosio, N.~Gigli, and G.~Savar{\'e}.
\newblock {\em Gradient flows: in metric spaces and in the space of probability
  measures}.
\newblock Springer Science \& Business Media, 2008.

\bibitem[AN00]{Amari2000}
S.~Amari and H.~Nagaoka.
\newblock {\em Methods of Information Geometry}, volume 191 of {\em
  Translations of Mathematical Monographs}.
\newblock American Mathematical Society, Providence, RI, USA, 2000.

\bibitem[AST18]{ambrosio2018pde}
Luigi Ambrosio, Federico Stra, and Dario Trevisan.
\newblock A pde approach to a 2-dimensional matching problem.
\newblock {\em Probability Theory and Related Fields}, pages 1--45, 2018.

\bibitem[BCP18]{BCP18}
J.~Bigot, E.~Cazelles, and N.~Papadakis.
\newblock Data-driven regularization of wasserstein barycenters with an
  application to multivariate density registration.
\newblock {\em ArXiv preprint: 1804.08962}, 2018.

\bibitem[BFS12]{burger2012regularized}
M.~Burger, M.~Franek, and C.-B. Sch{\"o}nlieb.
\newblock Regularized regression and density estimation based on optimal
  transport.
\newblock {\em Applied Mathematics Research eXpress}, 2012(2):209--253, 2012.

\bibitem[BGKL18]{BGKL18}
J.~Bigot, R.~Gouet, T.~Klein, and A.~Lopez.
\newblock Upper and lower risk bounds for estimating the {W}asserstein
  barycenter of random measures on the real line.
\newblock {\em to appear in Electronic Journal of Statistics}, 2018.

\bibitem[BL14]{W1}
S.~Bobkov and M.~Ledoux.
\newblock One-dimensional empirical measures, order statistics and
  {K}antorovich transport distances.
\newblock {\em To appear in the Memoirs of the American Mathematical Society},
  2014.

\bibitem[Bra06]{braides2006handbook}
A.~Braides.
\newblock A handbook of $\gamma$-convergence.
\newblock {\em Handbook of {D}ifferential {E}quations: stationary partial
  differential equations}, 3:101--213, 2006.

\bibitem[Caf92]{caffarelli1992}
L.~A Caffarelli.
\newblock The regularity of mappings with a convex potential.
\newblock {\em Journal of the American Mathematical Society}, 5(1):99--104,
  1992.

\bibitem[Caf96]{caffarelli1996}
L.~A Caffarelli.
\newblock Boundary regularity of maps with convex potentials--ii.
\newblock {\em Annals of mathematics}, 144(3):453--496, 1996.

\bibitem[CD14]{cuturi2013fast}
M.~Cuturi and A.~Doucet.
\newblock Fast computation of {W}asserstein barycenters.
\newblock In {\em International Conference on Machine Learning 2014, PMLR
  W\&CP}, volume~32, pages 685--693, 2014.

\bibitem[Cla13]{Clarke}
F.~Clarke.
\newblock {\em Functional analysis, calculus of variations and optimal
  control}, volume 264.
\newblock Springer Science \& Business Media, 2013.

\bibitem[CPSV18]{chizat2016interpolating}
L.~Chizat, G.~Peyr{\'e}, B.~Schmitzer, and F.-X. Vialard.
\newblock An interpolating distance between optimal transport and
  {F}isher--{R}ao metrics.
\newblock {\em Foundations of Computational Mathematics}, 18(1):1--44, 2018.

\bibitem[DE97]{MR1431744}
P.~Dupuis and R.~S. Ellis.
\newblock {\em A weak convergence approach to the theory of large deviations}.
\newblock Wiley Series in Probability and Statistics. John Wiley \& Sons, Inc.,
  New York, 1997.

\bibitem[DPF14]{de2014monge}
G.~De~Philippis and A.~Figalli.
\newblock The {M}onge--{A}mp{\`e}re equation and its link to optimal
  transportation.
\newblock {\em Bulletin of the American Mathematical Society}, 51(4):527--580,
  2014.

\bibitem[FG10]{figalli2010new}
A.~Figalli and N.~Gigli.
\newblock A new transportation distance between non-negative measures, with
  applications to gradients flows with {D}irichlet boundary conditions.
\newblock {\em Journal de math{\'e}matiques pures et appliqu{\'e}es},
  94(2):107--130, 2010.

\bibitem[FG15]{fournier:hal-00915365}
N.~Fournier and A.~Guillin.
\newblock On the rate of convergence in {W}asserstein distance of the empirical
  measure.
\newblock {\em Probability Theory and Related Fields}, 162(3-4):707--738, 2015.

\bibitem[FPPA14]{ferradans2014regularized}
S.~Ferradans, N.~Papadakis, G.~Peyr{\'e}, and J.-F. Aujol.
\newblock Regularized discrete optimal transport.
\newblock {\em SIAM Journal on Imaging Sciences}, 7(3):1853--1882, 2014.

\bibitem[Fr{\'e}48]{fre}
M.~Fr{\'e}chet.
\newblock Les {\'e}l{\'e}ments al{\'e}atoires de nature quelconque dans un
  espace distanci{\'e}.
\newblock {\em Annales de l'Institut H.Poincar{\'e}, Sect. B, Probabilit\'{e}s
  et Statistiques}, 10:235--310, 1948.

\bibitem[KP17]{KimPass}
Y.-H. Kim and B.~Pass.
\newblock {W}asserstein barycenters over {R}iemannian manifolds.
\newblock {\em Advances in Mathematics}, 307:640--683, 2017.

\bibitem[KT59]{kolmogorov1959varepsilon}
A.~N. Kolmogorov and V.~M. Tikhomirov.
\newblock $\varepsilon$-entropy and $\varepsilon$-capacity of sets in function
  spaces.
\newblock {\em Uspekhi Matematicheskikh Nauk}, 14(2):3--86, 1959.

\bibitem[LGL16]{gouic2015existence}
T.~Le~Gouic and J.-M. Loubes.
\newblock {Existence and Consistency of {W}asserstein Barycenters}.
\newblock {\em Probability Theory and Related Fields}, 168(3):901--917, 2016.

\bibitem[Ngu13]{nguyen2011wasserstein}
X.~L. Nguyen.
\newblock Convergence of latent mixing measures in finite and infinite mixture
  models.
\newblock {\em The Annals of Statistics}, 41(1):370--400, 2013.

\bibitem[Pas13]{pass2013optimal}
B.~Pass.
\newblock Optimal transportation with infinitely many marginals.
\newblock {\em Journal of Functional Analysis}, 264(4):947--963, 2013.

\bibitem[PZ16]{Pana15}
V.~M. Panaretos and Y.~Zemel.
\newblock Amplitude and phase variation of point processes.
\newblock {\em Annals of Statistics}, 44(2):771--812, 2016.

\bibitem[PZ17]{Pana17}
V.~M. Panaretos and Y.~Zemel.
\newblock Fr\'{e}chet means and {P}rocrustes analysis in {W}asserstein space.
\newblock {\em Bernoulli}, To be published, 2017.

\bibitem[Roc74]{rockafellar1974conjugate}
R.T. Rockafellar.
\newblock {\em Conjugate duality and optimization}.
\newblock SIAM, 1974.

\bibitem[San15]{santambrogio2015optimal}
Filippo Santambrogio.
\newblock Optimal transport for applied mathematicians.
\newblock {\em Birk{\"a}user, NY}, pages 99--102, 2015.

\bibitem[VDVW96]{van1996weak}
A.W. Van Der~Vaart and J.~A. Wellner.
\newblock {\em Weak convergence and empirical processes}.
\newblock Springer, 1996.

\bibitem[Vil03]{villani2003topics}
C.~Villani.
\newblock {\em Topics in optimal transportation}, volume~58 of {\em Graduate
  Studies in Mathematics}.
\newblock American Mathematical Society, 2003.

\bibitem[Vil08]{villani2008optimal}
C.~Villani.
\newblock {\em Optimal transport: old and new}, volume 338.
\newblock Springer Science \& Business Media, 2008.

\end{thebibliography}

\end{document}